\documentclass[11pt]{amsart}
\pdfoutput=1

\usepackage[margin=1.5in]{geometry}
\usepackage{mathptmx}
\usepackage{etex}
\usepackage[labelfont=bf]{caption}
\usepackage[section]{placeins}
\usepackage{url}
\usepackage{microtype}
\usepackage{longtable}
\usepackage{booktabs, comment}
\usepackage[french, english]{babel}
\usepackage{amsmath,amssymb,amsthm, comment, mathrsfs, url, yfonts}
\usepackage{array}
\usepackage[colorlinks]{hyperref}
\hypersetup{citecolor=blue}
\usepackage{amscd}
\usepackage{xypic}
\usepackage{todonotes}

\theoremstyle{definition}
\newtheorem{rem}{Remark}

\theoremstyle{plain}
\newtheorem{thm}{Theorem}
\newtheorem{prop}[thm]{Proposition}
\newtheorem{cor}[thm]{Corollary}
\newtheorem{conj}[thm]{Conjecture}

\newtheorem{lem}[thm]{Lemma}

\newcommand{\field}[1]{\mathbf{#1}}
\newcommand{\Qbar}{{\kern.1ex\overline{\kern-.1ex\Q\kern-.1ex}\kern.1ex}}
\newcommand{\overbar}[1]{{\kern.2ex\overline{\kern-.2ex#1\rule{0pt}{2mm}}}}
\newcommand{\Q}{\field{Q}}
\newcommand{\C}{\field{C}}
\newcommand{\R}{\field{R}}

\newcommand{\sH}{\mathcal{H}}
\newcommand{\sA}{\mathcal{A}}
\newcommand{\sM}{\mathcal{M}}

\newcommand{\Proj}{\mathbf{P}}

\newcommand{\Z}{\field{Z}}

\newcommand{\F}{\field{F}}
\newcommand{\Fbar}{\overline{\F\rule{0pt}{2.7mm}}}

\newcommand{\T}{\field{T}}

\newcommand{\CO}{\mathcal{O}}

\newcommand{\gp}{\mathfrak{p}}

\newcommand{\gN}{\mathfrak{N}}

\newcommand{\gm}{\mathfrak{m}}

\newcommand{\I}{\mathrm{I}}
\newcommand{\II}{\mathrm{II}}
\newcommand{\III}{\mathrm{III}}

\newcommand{\Tr}{\mathrm{Tr}}

\DeclareMathOperator{\Gal}{Gal}
\DeclareMathOperator{\GL}{GL}
\DeclareMathOperator{\SL}{SL}
\DeclareMathOperator{\End}{End}

\title[Abelian surfaces with good reduction]{\bf Examples of abelian
  surfaces with everywhere good reduction}

\author{Lassina Demb\'el\'e}
\address{Warwick Mathematics Institute, University of Warwick,
  Coventry CV4 7AL, United Kingdom}
\email{lassina.dembele@gmail.com}

\author{Abhinav Kumar}
\address{Department of Mathematics, Massachusetts Institute of
  Technology, 77 Massachusetts Institute of Technology, Cambridge MA
  02139, USA}
\email{abhinavk@alum.mit.edu}

\thanks{AK was supported in part by National Science Foundation grant
  DMS-0952486 and by a grant from the Solomon Buchsbaum Research
  Fund. LD is supported by the grant EPSRC EP/J002658/1.}
\begin{document}


\date{August 28, 2015}

\begin{abstract} 
We describe several explicit examples of simple abelian surfaces over
real quadratic fields with real multiplication and everywhere good
reduction. These examples provide evidence for the Eichler--Shimura
conjecture for Hilbert modular forms over a real quadratic
field. Several of the examples also support a conjecture of Brumer and
Kramer on abelian varieties associated to Siegel modular forms with
paramodular level structures.
\end{abstract}

\maketitle

\section{\bf \large Introduction}
A celebrated result of Fontaine~\cite{fontaine85} (see also
Abrashkin~\cite{abrashkin76}) asserts that there is no abelian scheme
over $\Z$. In other words, there is no abelian variety over $\Q$ with
everywhere good reduction. However, long before this result, there
were a few examples of elliptic curves of unit conductor over
quadratic fields in the literature. For example, the curve
\[
E:\, y^2+xy +\epsilon^2 y=x^3,
\]
where $\epsilon=\frac{5+\sqrt{29}}{2}$ is the fundamental unit in
$F=\Q(\sqrt{29})$, was known to Tate, and to Serre who extensively
studied it in~\cite{serre72}; it is also alluded to
in~\cite{shimura72}. Since then, there has been much work on finding
elliptic curves with everywhere good reduction over number fields,
with a particular emphasis on quadratic fields; see for
example~\cite{setzer81, stroeker83, comalada90, cremona92, pinch,
  kagawa98}. For real quadratic fields, the database of such curves
has been considerably expanded by
Elkies~\cite{elkies-db}. In~\cite{dembele-voight12}, it is shown that
this database is complete for all fundamental discriminants $\le 1000$
of narrow class number one, if one assumes modularity. A more
systematic algorithm which, given a number field $F$ and a finite set
of primes $S$ of its ring of integers, returns the set of all elliptic
curves over $F$ with good reduction outside of $S$ is given
in~\cite{cremona-lingham07}. However, this algorithm relies on
algorithms for $S$-integral points for elliptic curves, and has not
yet been full implemented for this reason. An alternate (and perhaps
more efficient) approach which uses $S$-unit equations is currently
being explored by Cremona and Elkies~\cite{cremona-elkies-project}. In
fact, a similar method has already been used by Smart~\cite{smart97}
to find hyperelliptic curves of genus $2$ with good reduction outside
$S$ when $F=\Q$.

In contrast, to the best of our knowledge there is not a single
example of an abelian surface with everywhere good reduction in the
literature (except in the case when the abelian surface has complex
multiplication \cite{dembele-donnelly}, or is a
$\Q$-surface~\cite{casselman70, shimura71} or a product of elliptic
curves).  This could possibly be explained by the fact that all the
algorithms we mentioned above do not readily generalize to the genus
$2$ situation. The goal of this paper is to remedy that situation by
providing the first equations for such surfaces over real quadratic
fields.

We note that the non-existence of abelian varieties over $\Q$ with
good reduction everywhere is instrumental in the Khare--Wintenberger
proof of the Serre conjecture for Galois representations of
$\Gal(\Qbar/\Q)$. As described in \cite{khare-motives}, the proof of
the Serre conjecture in retrospect can be viewed as a method to
exploit an accident which occurs in three different guises:

\begin{enumerate}
\item[(a)] (Fontaine, Abrashkin) There are no non-zero abelian
  varieties over $\Z$.
\item[(b)] (Serre, Tate) There are no irreducible representations
\[
\overbar{\rho}: \Gal(\Qbar/\Q)\to\GL_2(\Fbar),
\]
where $\Fbar$ is the algebraic closure of $\F_2$ or $\F_3$,
that are unramified outside of $2$ and $3$ respectively.
\item[(c)] $S_2(\mathrm{SL}_2(\Z))=0$, i.e., there are no cusp forms
  of level $\mathrm{SL}_2(\Z)$ and weight $2$.
\end{enumerate}
The failure of this happy accident over general number fields, such as
real quadratic fields, means that new techniques are needed for
analogous modularity results.

Our approach to the construction of abelian surfaces with everywhere
good reduction combines three key elements: (a) recent advances in the
computation of Hecke eigenvalues of Hilbert modular forms, (b) new
rational models of Hilbert modular surfaces, and (c) the
Eichler--Shimura conjecture for Hilbert modular forms. As a result of
our investigation, we produce further evidence for the
Eichler--Shimura conjecture, as well as for a conjecture of Brumer and
Kramer \cite{brumer-kramer12} associating abelian varieties to
paramodular Siegel modular forms on $\mathrm{Sp}(4)$.

The outline of the paper is as follows: in Sect.~\ref{sec:background},
we briefly recall the basic facts regarding these three
ingredients. In Sect.~\ref{sec:strategy}, we describe our strategy to
predict and find examples of good reduction abelian surfaces, assuming
the Eichler--Shimura conjecture. Section~\ref{sec:examples} provides
several illustrative examples of our methods, giving explicit abelian
surfaces with good reduction everywhere, and connecting them to
appropriate Hilbert modular forms. We conclude with a list of all our
examples in Sect.~\ref{sec:data}.

\section{\bf \large Background} \label{sec:background}

\subsection{\bf Hilbert modular forms} \label{subsec:hmf}

Let $F$ be a totally real field of narrow class number one and degree
$d$. We let $\CO_F$ be the ring of integers of $F$, $\mathfrak{d}_F$
the different of $F$. For each $i=1,\ldots, d$, let $a\mapsto a^{(i)}$
denote the $i$-th embedding of $F$ into $\R$, so that we have an
identification $F\otimes \R \simeq \R^d$. We let $F_{+}$ be the set of
totally positive elements in $F$, i.e. the inverse image of
$(\R_+)^d$, and $\CO_{F,+} = F_{+} \cap \CO_F$. We fix a totally
positive generator $\delta$ of $\mathfrak{d}_F$. (Note that every
ideal has such a generator since $F$ has narrow class number one.)

Let $\sH$ be the Poincar\'e upper half plane. The Hilbert modular
group $\SL_2(\CO_F)$ acts on $\sH^d$ by fractional linear
transformations:
\[
\begin{pmatrix} a & b \\ c & d \end{pmatrix}  \cdot
(z_1, \dots, z_d) = \left( \frac{a^{(i)} z_i + b^{(i)}}{c^{(i)} z_i +
  d^{(i)}} \right)_{i=1,\dots,d} \hspace{-10mm}.
\]
Let $\gN$ be an integral ideal, and set
\[
\Gamma_0(\mathfrak{N}) = \left\{  \begin{pmatrix} a & b \\ c & d \end{pmatrix} \in \SL_2(\CO_F) : c \in \mathfrak{N} \right\}.
\]
Let $k\ge 2$ be an even integer. A Hilbert modular form of
weight\footnote{More precisely, this defines a Hilbert modular form of
  {\em parallel} weight $k$.} $k$ and level $\gN$ is a holomorphic
function $f: \sH^d \to \C$ such that
\[
f(\gamma z) = \Big( \prod_{i=1}^d (c^{(i)}z_i + d^{(i)})\Big)^k f(z)\,\,\text{for all}\,\,\gamma = \begin{pmatrix} a & b \\ c & d \end{pmatrix}
\in \Gamma_0(\gN).
\]

Let $f$ be a Hilbert modular form of weight $k$ and level $\gN$.  Then
$f$ is invariant under the matrices $\left(\!\begin{smallmatrix}1 &
\mu \\ 0 & 1\end{smallmatrix}\!\right)$ for $\mu \in \CO_F$, which act
as $z\mapsto z + \mu$.  Hence, by the Koecher
principle~\cite{bruinier}, $f$ admits a $q$-expansion of the form
\[
f(z) = a_0 +\sum_{\mu \in \CO_{F,+}} a_\mu e^{2\pi i \Tr(\frac{\mu
    z}{\delta})},
\]
where $\Tr(\nu z) = \nu^{(1)} z_1 + \cdots + \nu^{(d)} z_d,$ for $\nu
\in F_{+}$.  We say that $f$ is a cusp form if $a_0=0$. Since $f$ is
invariant under the action of the matrices $\mathrm{diag}(\epsilon,
\epsilon^{-1})$ for $\epsilon \in \CO_F^\times$ in $\SL_2(\CO_F)$,
which act as $z\mapsto \epsilon^2 z$, we have $a_{\epsilon^2\mu} =
a_\mu$ for all $\mu \in \CO_{F, +}$ and $\epsilon \in
\CO_F^\times$. Let $f$ be a cusp form of weight $k$ and level
$\gN$. Then, for every ideal $\gm \subseteq \CO_F$, the quantity
$a_{\gm}(f) = a_{\mu}$, where $\mu$ is a totally positive generator of
$\gm$, is well-defined and depends only on $\gm$. We call it the
$\gm$-th Fourier coefficient of $f$. When $f$ is a normalized
eigenform for the Hecke operators (i.e. $a_{(1)}(f) = 1$), the
eigenvalue of the Hecke operator $T_{\gm}$ is $a_{\gm}(f)$ for each
$\gm \nmid \gN$. It is a theorem of Shimura~\cite{shimura78} that in
this situation, the $a_{\gm}(f)$ are algebraic integers and the
$\Z$-subalgebra $\CO_f = \Z[a_{\gm}(f): \gm \subseteq\CO_F]\subset \C$
has finite rank and is therefore an order in some number field $K_f$
(called the field of Fourier coefficients of $f$). For more background
on Hilbert modular forms, see \cite{dembele-voight12, freitag,
  bruinier}.

Here, we wish to point out some new techniques in the computation of
Hilbert modular forms, which arise from the
Eichler--Jacquet--Langlands--Shimizu correspondence between Hilbert
modular forms and quaternionic modular forms. We will not go into
details here, but instead refer the reader to \cite{dembele-voight12}
for a detailed description of these methods. The upshot is that it is
possible to efficiently compute systems of Hecke eigenvalues for
Hilbert modular cusp forms by instead computing modular forms on
finite spaces or on Shimura curves. This will be crucial to our
methods in this paper. The corresponding algorithms have been
implemented in the Hilbert Modular Forms Package in \texttt{Magma}
\cite{magma-pkg}).

\subsection{\bf Hilbert modular surfaces} \label{subsec:hms}

Let $K$ be a real quadratic field of discriminant $D'$. The Hilbert
modular surface $Y_{-}(D')$ is a compactification of the coarse moduli
space which parametrizes principally polarized abelian surfaces with
real multiplication by the ring of integers $\CO_{K}$ of $K$,
i.e. pairs $(A, \iota)$, where $\iota: \CO_{K} \to \End_{\Qbar}(A)$ is
a homomorphism. The complex points $Y_{-}(D')(\C)$ of this space are
obtained by compactifying $\SL_2(\CO_{K}) \backslash (\sH^+ \times
\sH^-$), where $\sH^+$ and $\sH^-$ are the upper and lower half-planes
respectively, by adding finitely many cusps and resolving the
singularities of the resulting space. The Hilbert modular surface maps
to the moduli space $\sA_2$ of principally polarized abelian surfaces,
by forgetting the action of $\CO_K$. Its image is the Humbert surface
$\sH_{D'}$, and the map $Y_{-}(D') \to \sH_{D'}$ is a double cover,
ramified along a union of modular curves. The surfaces $Y_{-}(D')$
have models over the integers, with good reduction away from primes
dividing $D'$.

Recently, Elkies and the second author \cite{elkies-kumar12} computed
explicit birational models over $\Q$ for these Hilbert modular
surfaces for all the fundamental discriminants $D'$ less than $100$,
by identifying the Humbert surface $\sH_{D'}$ with a moduli space of
elliptic K3 surfaces, which may be computed explicitly. For the
fundamental discriminants in the range $1 < D' < 100$, the Humbert
surface is a rational surface, i.e. birational to $\Proj^2$ over
$\Qbar$ (and in fact, even over $\Q$). Therefore, they are able to
exhibit $\sH_{D'}$ as a double cover of $\Proj^2$, with equation $z^2
= f(r,s)$, where $r,s$ are parameters on $\Proj^2$. They also get the
map to $\sA_2$, which is birational to $\sM_2$, the moduli space of
genus $2$ curves. It is given by producing the Igusa--Clebsch
invariants of the image point as rational functions of $r$ and $s$.

\begin{rem}
Hilbert modular surfaces have been an object of extensive study in
number theory and arithmetic geometry, especially in the latter half
of the twentieth century. In particular, their geometric
classification was described by Hirzebruch, van de Ven, Zagier, van
der Geer and others. In the comprehensive reference
\cite{van-der-Geer}, arithmetic models for some of them are also
described. However, for our work, we need explicit equations for these
surfaces along with the map to $\sA_2$, and this does not seem to be
available for any discriminant other than $5$ except in
\cite{elkies-kumar12} (though it could be worked out in principle
using Hilbert and Siegel modular forms). Consequently, we will use the
equations from \cite{elkies-kumar12} throughout.
\end{rem}

\subsection{\bf Eichler--Shimura conjecture} \label{subsec:eichler-shimura}

The following conjecture is instrumental in identifying the examples
in this paper.

\begin{conj}[Eichler--Shimura]\label{conj:eichler-shimura} 
Let $F$ be a totally real number field and $\gN$ an integral ideal of
$F$. Let $f$ be a Hilbert newform of weight $2$ and level $\gN$. Let
$\CO_f = \Z[a_\gm(f): \gm \subseteq \CO_F]$ be the order generated by
the Fourier coefficients of $f$, and $K_f$ its field of
fractions. There exists an abelian variety $A_f/F$ of dimension
$[K_f:\Q]$ with good reduction outside of $\gN$ and with $\CO_f
\hookrightarrow \mathrm{End}_F(A_f)$, such that
\[
L(A_f,s)=\prod_{\tau\in\mathrm{Hom}(K_f, \C)}L(f^\tau, s),
\] 
where
\[
L(f^\tau, s):=\sum_{\gm\subseteq \CO_F}\frac{a_{\gm}(f)^\tau}{\mathrm{N}\gm^s}.
\]
\end{conj}

When $F=\Q$, this conjecture is a theorem, due to Eichler for prime
level and Shimura in the general case. The Eichler--Shimura
construction can be summarized as follows. Let $N>1$ be an integer,
and let $X_1(N)$ be the modular curve of level $\Gamma_1(N)$. This
curve and its Jacobian $J_1(N)$ are defined over $\Q$. We recall that
the space $S_2(\Gamma_1(N))$ of cusp forms of weight $2$ and level
$\Gamma_1(N)$ is a $\T$-module, where $\T$ is the Hecke algebra. Let
$f \in S_2(\Gamma_1(N))$ be a newform, and let
$I_f=\mathrm{Ann}_{\T}(f)$. Shimura~\cite{shimura73} showed that the
quotient
\[
A_f:=J_1(N)/I_fJ_1(N)
\]
is an abelian variety $A_f$ of dimension $[K_f:\Q]$ defined over $\Q$
with endomorphisms by the order $\CO_f = \Z[a_n(f): n\ge 1]$ and that
\[
L(A_f, s)=\prod_{g \in [f]} L(g, s),
\]
where $[f]$ denotes the Galois orbit of $f$.

One of the main consequences of the proof of the Serre
conjecture~\cite{serre87} by
Khare--Wintenberger~\cite{khare-wintenberger09} is that the converse
to Conjecture~\ref{conj:eichler-shimura} is true when $F=\Q$. That is,
an abelian variety of $\GL_2$-type is isogenous to a $\Q$-simple
factor of $J_1(N)$ for some $N$ \cite{khare-wintenbergerICM}. And so,
this provides a theoretical construction of all abelian varieties of
$\GL_2$-type over $\Q$ with a prescribed conductor. In fact, one can
make this explicit in many cases (see~\cite{cremona97} for elliptic
curves, and~\cite{gonzalez-guardia02, guardia02} for abelian
surfaces).

For $[F:\Q]>1$, the known cases of
Conjecture~\ref{conj:eichler-shimura} exploit the cohomology of
Shimura curves. For instance, the conjecture is known when $[F:\Q]$ is
odd, or when $\mathfrak N$ is exactly divisible by a prime $\mathfrak
p$ of $\CO_F$ \cite{zhang}. The simplest case in which
Conjecture~\ref{conj:eichler-shimura} is still unknown is when $f$ is
a newform of level $(1)$ and weight $2$ over a real quadratic
field. In that case, the conjecture predicts that the associated
abelian variety $A_f$ has everywhere good reduction.

\section{\bf \large The strategy} \label{sec:strategy}
Let $F$ be a number field of class number $1$, and $E$ an elliptic
curve over $F$ given by a (global minimal) Weierstrass equation
\[
E: y^2 + a_1xy +a_3y =x^3+a_2 x^2 +a_4 x + a_6,
\] 
with $a_i\in \CO_F$, the ring of integers of $F$. The invariants $c_4$
and $c_6$ of $E$ satisfy the equation $c_4^3-c_6^2=1728\Delta$, where
$\Delta$ is the discriminant of $E$. In other words, the pair $(c_4,
c_6)$ is an $\CO_F$-integral point on the curve
\begin{align}\label{eq:ell-curve}
y^2=x^3-1728\Delta.
\end{align}
Since $E$ has everywhere good reduction if and only if $\Delta$ is a
unit in $\CO_F$, we can find all the elliptic curves over $F$ with
everywhere good reduction by solving~(\ref{eq:ell-curve}) as $\Delta$
runs over a finite set of representatives of
$\CO_F^\times/(\CO_F^\times)^{12}$. (Note that given a pair $(c_4,
c_6)$, we get a minimal model by using the Tate algorithm.) Most of
the algorithms we mentioned earlier rely on this fact.

Unfortunately, the reduction of abelian varieties of higher dimension
is not characterized by a nice single diophantine equation such
as~(\ref{eq:ell-curve}). For this reason, we need an additional
ingredient which will guide our search. This extra input is provided
by the Eichler--Shimura conjecture.

Suppose we have a Hilbert modular eigenform $f$ of weight $2$ over
$F$, with Hecke eigenvalues $a_{\gm}(f)$ in a real quadratic field
$K_f$ of discriminant $D'$. The Eichler--Shimura conjecture predicts
that there should be an abelian variety $A$ over $F$ of dimension
$[K_f:\Q]= 2$, (up to isogeny) associated to this data, which has real
multiplication by an order in $K_f$. Furthermore, the conductor of $A$
should divide the level $\gN$ of $f$. In particular, if $f$ has level
$(1)$, the conjectural abelian surface $A$ has good reduction
everywhere. This observation will be the source of our examples in
this paper, for which the abelian surface turns out to be principally
polarized, and also has real multiplication by the full ring of
integers of $K_f$. Our strategy to produce such $A$ is as follows:

\begin{enumerate}
\item[(a)] Find a Hilbert modular form of level $(1)$ and weight $2$
  for a real quadratic field $F$, with coefficients in a real
  quadratic field $K_f$ of discriminant $D'$.
\item[(b)] Find an $F$-rational point on the Hilbert modular surface
  $Y_{-}(D')$, for which the $L$-function of the associated abelian
  surface matches that of $f$ at several Euler factors, up to twist.
\item[(c)] Compute the correct quadratic twist of the abelian surface,
  or the genus $2$ curve.
\item[(d)] Check that the abelian surface has good reduction
  everywhere.
\item[(e)] Prove that the $L$-functions indeed match up.
\end{enumerate}

Note that there is no reason one has to restrict to the case when the
base field is a real quadratic field $F$. The next interesting case in
which the Eichler--Shimura conjecture is not known is that of totally
real quartic base fields $L$. So one could look for eigenforms of
weight $2$ for $\SL_2(\CO_L)$ whose Fourier coefficients are in a real
quadratic field $K$ of discriminant $D$, and on the other hand try to
find $L$-rational points on the Hilbert modular surface $Y_{-}(D)$. In
this paper, we looked at quadratic base fields $F$ for convenience. On
the other hand, if we instead want examples for which the field $K_f$
has larger degree, we might need explicit rational models for the
appropriate Hilbert modular varieties, which are not currently
available. Hence the choice of $K$ is restricted.

For simplicity, we investigated only real quadratic fields $F$ of
narrow class number $1$ and discriminant less than $1000$. We found
twenty-eight examples of Hilbert newforms, and corresponding abelian
surfaces for most of these forms. We will say a few words later about
the ``missing'' examples, which we hope will be found in future work.

\section{\bf \large The examples}\label{sec:examples}
From now on, $F$ will denote a real quadratic field of narrow class
number one. We let $D$ be its fundamental discriminant. We will denote
its ring of integers by $\CO_F$. Let $w = \sqrt{D}$ or $(1 +
\sqrt{D})/2$ according as $D$ is $0$ or $1$ mod $4$, so that $\{1,w\}$
is a $\Z$-basis of $\CO_F$. For a Hilbert newform $f$ of weight $2$
over $F$, we will let $\CO_f = \Z[a_{\gm}(f): \gm \subseteq \CO_F]$
and $K_f$ be the order and the field generated by the Fourier
coefficients, respectively. We will focus on forms such that
$[K_f:\Q]=2$, since we do not yet know how to write simple equations
for general Hilbert modular varieties. We let $D'$ be the discriminant
of $K_f$ and write $e = \sqrt{D'}$ or $(1 + \sqrt{D'})/2$.  We denote
the non-trivial element of $\Gal(F/\Q)$ and $\Gal(K_f/\Q)$ by $\sigma$
and $\tau$ respectively. The $L$-series of the conjectural surface
$A_f$ attached to $f$ is written as
\[
L(A_f, s) = L(f, s) L(f^\tau, s)= \prod_{\gp} \frac{1}{Q_{\gp}(\mathrm{N}(\gp)^{-s})},
\]
where
\begin{align*}
Q_{\gp}(T) &:=\big(T^2-a_{\gp}(f)T+\mathrm{N}(\gp)\big)\big(T^2-a_{\gp}(f)^\tau T+\mathrm{N}(\gp)\big)\\
&\,\, = T^4 - s_{\gp}(f) T^3 + t_{\gp}(f)T^2 - \mathrm{N}(\gp) s_{\gp}(f)T + \mathrm{N}(\gp)^2. 
\end{align*}

Our examples (Table~\ref{tab1}) can be subdivided in the following
cases, with the majority of examples coming from Case II.
\begin{enumerate}
\item[\bf I:] The form $f$ is $\Gal(F/\Q)$-invariant.
\item[\bf II:] The form $f$ is not $\Gal(F/\Q)$-invariant, but its
  $\Gal(K_f/\Q)$-orbit $\{ f, f^\tau\}$ is.
\item[\bf III:] The $\Gal(K_f/\Q)$-orbit $\{f,f^\tau\}$ is not
  $\Gal(F/\Q)$-invariant.
\end{enumerate}

\begin{table} [h!] 
\caption{A summary of the examples}
\label{tab1}
\begin{tabular}{ccc}
{
\begin{tabular}{ >{$}c<{$}   >{$}c<{$}   >{$}c<{$}   }
\toprule
D & D' & \text{Case} \\
\midrule
53 & 8 & \I \\ 
61 & 12 & \I\\ 
73& 5 & \I \\ 
193 & 17 & \II \\ 
233&17 & \II\\ 
277&29 & \II\\ 
349&21 & \II\\
353& 5 & \III\\
373&93 & \II\\
389&8 & \II\\
\bottomrule
\end{tabular}

\smallskip} &

{\begin{tabular}{ >{$}c<{$}   >{$}c<{$}   >{$}c<{$}   }
\toprule
D & D' & \text{Case} \\
\midrule
397&24 & \II\\
409&13 & \II\\
421&5 & \I \\
421 & 5 & \III \\
433&12 & \II\\
461&29 & \II\\
613&21 & \II \\
677&13 & \II \\
677 & 29 & \II \\
 & & \\
\bottomrule
\end{tabular}

\smallskip} &

{ \begin{tabular}{ >{$}c<{$}   >{$}c<{$}   >{$}c<{$}   }
\toprule
D & D' & \text{Case} \\
\midrule
677 & 85 & \II \\
709&5 & \II \\
797&8 & \II \\
797 & 29 & \II\\
809&5 & \II \\
821&44 & \II \\
853&21 & \II\\
929&13 & \II \\
997&13 & \II \\
& & \\
\bottomrule
\end{tabular}

\smallskip}

\end{tabular}
\end{table}

We will see that Case I is somewhat special: it is frequently possible
to produce the associated abelian surface through analytic methods for
classical modular forms.

In~\cite{brumer-kramer12}, Brumer--Kramer proposed the following
conjecture as a genus $2$ analogue of the Eichler--Shimura
construction for classical newforms of weight $2$ (with integer
coefficients).

\begin{conj}[Brumer--Kramer]\label{conj:brumer-kramer}
Let $g$ be a paramodular Siegel newform of genus $2$, weight $2$ and
level $N$, with integer Hecke eigenvalues, which is not in the span of
Gritsenko lifts.  Then there exists an abelian surface $B$ defined
over $\Q$ of conductor $N$ such that $\End_{\Q}(B)=\Z$ and $L(g,
s)=L(B,s)$.
\end{conj}

The examples in Case II show that there is a strong connection between
this conjecture and Conjecture~\ref{conj:eichler-shimura}.

\subsection{Case I} 

In this case, the Hecke eigenvalues of the Hilbert modular form $f$
satisfy 
\[
a_\gp(f) = a_{\gp^\sigma}(f).
\] 
This implies that the form $f$ is a base change from $\Q$.  Let $g$ be
a newform in $S_2(\Gamma_1(D))$ whose base change is $f$. Since the
level of $f$ is $(1)$, the form $g \in S_2(\Gamma_1(D),
\chi_D)^{\textrm{new}}$ by~\cite[Prop. 2, p. 263]{mazur-wiles84},
where $\chi_D$ is the fundamental character of the quadratic field
$F=\Q(\sqrt{D})$. Let $L_g$ be the coefficient field of $g$. Then,
$L_g$ is a quartic CM field which contains $K_f$. The non-trivial
element of $\Gal(L_g/K_f)$, which we denote by $(x\mapsto
\overline{x}, x \in L_g)$, extends to complex conjugation.  Let $B_g$
be the abelian variety attached to the form $g$. Then $B_g$ is a
fourfold such that $\End_\Q(B_g)\otimes \Q \simeq L_g$. Let $w_D$ be
the Atkin--Lehner involution on $S_2(\Gamma_1(D),
\chi_D)^{\textrm{new}}$. This induces an involution on $B_g$, which we
still denote by $w_D$. Shimura~\cite[\S~7.7]{shimura71} shows the
following:

\begin{enumerate}
\item[(a)] $w_D$ is defined over $F$, and $w_D^\sigma = - w_D$;
\item[(b)] $w_D\cdot [x] = [\overline{x}]\cdot w_D$, where $[x]$
  denotes the endomorphism induced on $B_g$ by $x \in L_g$.
\item[(c)] The abelian surface $A_f := (1 + w_D) B_g$ is defined over
  $F$, and is isogenous to its Galois conjugate given by $A_f^\sigma
  := (1 - w_D) B_g$. Moreover, we have
\[
B_g\otimes_{\Q} F \sim A_f \times A_f^\sigma.
\]
\end{enumerate}
So in this case, the existence of the surface $A_f$ is a direct
consequence of the classical Eichler--Shimura construction.

Although Conjecture~\ref{conj:eichler-shimura} is known in this case,
it would still be desirable to have an explicit equation for the
surface $A_f$. We outline two methods to find it, the first of which
is special to this case.

\subsubsection{\bf Method 1} \label{subsec:method1}

This method is analytic, and has an obvious connection with the Oda
conjecture~\cite[p.~xii]{oda72} for Hilbert modular forms that arise
from base change. It assumes that both $A_f$ and $A_f^\sigma$ are
principally polarizable.  To describe it, we recall that by
\cite[Theorems 6.2.4 and 6.2.6]{couveignes-edixhoven2011}, there exist
newforms $g_1, g_2 \in S_2(\Gamma_1(D), \chi_D)^{\textrm{new}}$ such
that $g_1, \overline{g}_1$, $g_2$ and $\overline{g}_2$ form a basis of
the Hecke constituent of $g$ and
\[
w_D (g_1) = \overline{\lambda}_D(g_1)\overline{g}_1,\, w_D (g_2) =
\overline{\lambda}_D(g_2)\overline{g}_2,
\]
where $a_D(g)$ is the Hecke eigenvalue of $g$ at $D$ and
$\lambda_D(g)=\frac{a_D(g)}{\sqrt{D}}$, the pseudo-eigenvalue of
$w_D$. The matrix of $w_D$ in the basis $\{g_1, \overline{g}_1, g_2,
\overline{g}_2\}$ is given by
\begin{align*} W_D := 
\begin{bmatrix} 0&\lambda_D(g_1)&0&0\\
\overline{\lambda}_D(g_1)&0&0&0\\
0&0&0&\lambda_D(g_2)\\
0&0&\overline{\lambda}_D(g_2)&0
\end{bmatrix}
\end{align*}
From this, we see that $W_D^\sigma = -W_D$. The following lemma is a
simple adaptation of Cremona's~\cite[Lemma 5.6.2]{cremona92}.

\begin{lem}\label{lem:differential-forms} 
The set of forms $h_i^{\pm} := \frac{1}{2}(g_i \pm w_D(g_i))$,
$i=1,2$, are bases for the $\pm$-eigenspaces of $W_D$, acting on the
Hecke constituent of $g$, which give a decomposition of the space of
differential $1$-forms $H^0(B_g\otimes_\Q F, \Omega_{B_g\otimes_\Q
  F/F}^1)$ according to the action of $\Gal(F/\Q)$.
\end{lem}

Let $H_1(B_g, \Z)^{\pm}$ denote the $\pm$-eigenspaces of $w_D$. They
are free Hecke submodules of $H_1(B_g,\Z)$ of rank $4$ over $\Z$,
which are direct summands.

\begin{lem}\label{lem:shimura-decomposition} 
Let $\Lambda_g^{\pm}$ be the period lattices obtained by integrating
the forms in Lemma~\ref{lem:differential-forms} against $H_1(B_g,
\Z)^{\pm}$, and set $\Lambda_g = \Lambda_g^+ \oplus
\Lambda_g^-$. Then, there exist an abelian fourfold $B_g'$ defined
over $\Q$, and an isogeny $\phi: B_g' \to B_g$ whose degree is a power
of $2$, such that $B_g'(\C) = \C^4/\Lambda_g$. Moreover, $B_g' =
\mathrm{Res}_{F/\Q}(A_f)$ where $A_f$ is an abelian surface defined
over $F$.
\end{lem}

\begin{proof} 
We first note that the complex tori $\C^2/\Lambda_g^{\pm}$ and
$\C^4/\Lambda_g$ have canonical Riemann forms obtained by restriction
of the intersection pairing $\langle \cdot, \cdot\rangle$ on
$B_g$. Therefore, they are the complex points of some abelian
varieties. Since $h_1^{+}, h_2^{+}, h_1^{-}, h_2^{-}$ is a basis of
the Hecke constituent of $g$, \cite[Theorem~7.14 and
  Proposition~7.19]{shimura71} imply that there exist a fourfold
$B_g'$ defined over $\Q$, and an isogeny $\phi:\,B_g'\to B_g$, such
that $B_g'(\C) = \C^4/\Lambda_g$.

Next, let $x \in H_1(B_g, \Z)$, then we have
\[
2x = (x + w_D x ) + (x - w_D x) = y_{+} + y_{-} \in H_1(B_g, \Z)^+
\oplus H_1(B_g, \Z)^-.
\] 
Hence the exponent of $H_1(B_g, \Z)^+ \oplus H_1(B_g, \Z)^-$ inside
$H_1(B_g, \Z)$ divides $2$.  This implies that the degree of $\phi$ is
a power of $2$.

Since $w_D$ is defined over $F$ and $w_D^\sigma = - w_D$, the bases
$\{h_1^{+}, h_2^{+}\}$ and $\{h_1^{-}, h_2^{-}\}$ are
$\Gal(F/\Q)$-conjugate. Therefore $\C^2/\Lambda_g^+$ and
$\C^2/\Lambda_g^-$ are the complex points of some abelian surfaces
defined over $F$ that are Galois conjugate. Let $A_f$ be the surface
such that $A_f(\C) = \C^2/\Lambda_g^+$.  Then, we see that $B_g' =
\mathrm{Res}_{F/\Q}A_f$ by construction.
\end{proof}

In practice, we can replace $B_g$ by $B_g'$, and hence assume that
\[
H_1(B_g, \Z) = H_1(B_g, \Z)^+ \oplus H_1(B_g, \Z)^- = H_1(A_f, \Z)
\oplus H_1(A_f^\sigma, \Z).
\] 
The above integration then gives the period lattice decomposition
\[
\varOmega_{B_g} = \varOmega_{A_f} \times \varOmega_{A_f^\sigma} =
(\varOmega_1\,|\,\varOmega_2) \times
(\varOmega_1^\sigma\,|\,\varOmega_2^\sigma).
\]
Provided that the intersection pairing restricted to $H_1(A_f, \Z)$
and $H_1(A_f^\sigma, \Z)$ induces principal polarizations, we can
compute the surfaces $A_f$ and $A_f^\sigma$ as Jacobians of curves
$C_f$ and $C_f^\sigma$ (defined over $F$).

We illustrate this with the following example. The smallest
discriminant for which we obtain a surface which satisfies Case I is
$D=53$ (see Table \ref{table:eigenform-sqrt53}). The abelian surface
$A_f$ has real multiplication by (an order in) the field
$\Q(\sqrt{2})$. In fact, we will see that it has real multiplication
by the full ring of integers.

A symplectic basis for $H_1(B_g, \Z)$ is given by the modular
symbols~\cite{stein07}
\begin{small}
\begin{align*}
\gamma_1 &:= -\{-1/35, 0\} + \{-1/26, 0\},\\
\gamma_2 &:=  -\{-1/47, 0\},\\
\gamma_3 &:= \{-1/37, 0\},\\
\gamma_4 &:= \{-1/47, 0\}  - \{-1/15, 0\} + \{-1/13, 0\},\\
\gamma_5 &:= -\{-1/28, 0\},\\
\gamma_6 &:= -\{-1/44, 0\},\\
\gamma_7 &:= \{-1/15, 0\} - \{-1/44, 0\},\\
\gamma_8 &:= \{-1/28, 0\} + \{-1/21, 0\}  - \{-1/26, 0\}.
\end{align*}
\end{small}

Computing the matrix $G$ of the intersection pairing in that basis, we
see that $B_g$ is principally polarized. We obtain the integral bases
$\{\delta_1, \delta_2, \delta_3, \delta_4\}$ and $\{\delta_1',
\delta_2', \delta_3', \delta_4'\}$ for $H_1(B_g, \Z)^{+}$ and
$H_1(B_g, \Z)^{-}$, respectively, where
\begin{small}
\begin{align*}
\delta_1 &:= -\{-1/35, 0\} + \{-1/26, 0\},\\
\delta_2 &:= \{-1/37, 0\} - \{-1/47, 0\} + \{-1/15, 0\} - \{-1/13, 0\},\\
\delta_3 &:= - \{-1/28, 0\},\\
\delta_4 &:= - \{-1/28, 0\} + \{-1/15, 0\} - \{-1/44, 0\} - \{-1/21, 0\} + \{-1/26, 0\},\\
\delta_1' &:= -\{-1/47, 0\},\\
\delta_2' &:= \{-1/37, 0\} + \{-1/47, 0\}  - \{-1/15, 0\} + \{-1/13, 0\},\\
\delta_3' &:=-\{-1/44, 0\},\\
\delta_4' &:= \{-1/28, 0\} + \{-1/15, 0\} - \{-1/44, 0\} + \{-1/21, 0\}  - \{-1/26, 0\}.
\end{align*}
\end{small}

In this case, we verify that the index of $H_1(B_g, \Z)^+ \oplus
H_1(B_g, \Z)^-$ inside $H_1(B_g, \Z)$ is $4$, and that the restriction
of the intersection pairing to each direct summand $H_1(B_g,
\Z)^{\pm}$ is of type $(1, 2)$. This means that $A_f$ and $A_f^\sigma$
are not principally polarized with respect to the Riemann form given
by the restriction of the intersection pairing from $B_g$.  Let
$G^{\pm}$ be the corresponding matrices for these pairings. We remedy
this situation by finding a suitable element of the Hecke algebra, as
in \cite[Section 4.2]{gonzalez-guardia-rotger05}. The element $u =
-e-2 \in \CO_f$ has norm $2$, and acts on $H_1(B_g, \Z)^{\pm}$ as
$T_7^{\pm}$ where $T_7$ is the Hecke operator at $7$.  Letting
$G_u^{\pm}= T_7^{\pm}\cdot G^{\pm}$, we obtain principal polarizations
on $A_f$ and $A_f^\sigma$ by
\cite[Proposition~3.11]{gonzalez-guardia-rotger05}.

By integrating the bases of differential forms $\{h_1^{+}, h_2^{+}\}$
and $\{h_1^{-},h_2^{-}\}$ from Lemma~\ref{lem:differential-forms}
against the Darboux bases
\begin{align*}
\begin{pmatrix} \eta_1\\ \eta_2\\ \eta_3 \\ \eta_4\end{pmatrix} := 
\begin{pmatrix}0&0&1&2\\ 1 &-1 &0 &0\\ 0& 1 &0 &0\\ 0 &0 &0 &1 \end{pmatrix}
\begin{pmatrix} \delta_1\\ \delta_2\\ \delta_3 \\ \delta_4\end{pmatrix}\,\text{and}\,
\begin{pmatrix} \eta_1'\\ \eta_2'\\ \eta_3'\\ \eta_4'\end{pmatrix} := 
\begin{pmatrix} 1& 0& 0& 4\\ 0& 0& 0& 1\\ 0& 1& 0& 0\\ 0& 0& 1& 0\end{pmatrix}
\begin{pmatrix} \delta_1'\\ \delta_2'\\ \delta_3' \\ \delta_4'\end{pmatrix},
\end{align*}

respectively, we obtain the Riemann period matrices $\varOmega_{A_f}$
and $\varOmega_{A_f^\sigma}$, where
\begin{align*}
\varOmega_1 &:= \begin{pmatrix}
2.53595... + 2.39271...i& -4.32914... - 4.08462...i\\
-66.45185... - 24.43147...i& 19.46329... + 7.15581...i\\ 
\end{pmatrix},\\
\varOmega_2 &:= \begin{pmatrix}
1.79318... - 1.69190...i& 6.12233... - 5.77653...i\\
46.98855... - 17.27566...i& 27.52526... - 10.11984...i
\end{pmatrix},\\
\varOmega_1^\sigma &:= \begin{pmatrix}
-2.44814... + 4.22343...i& 2.44814... + 4.22343...i\\
 0.78506... + 1.10501...i& -0.78506... + 1.10501...i
\end{pmatrix},\\
\varOmega_2^\sigma &:= \begin{pmatrix}
1.43409... + 2.47403...i& -8.35849... + 14.41970...i\\
-2.68038... + 3.77277...i& 0.45988... + 0.64730...i
\end{pmatrix}. 
\end{align*}
This yields the normalized period matrices
\begin{align*}
Z &:= \begin{pmatrix}
-0.65878... + 0.69909...i& -0.40996... + 0.82303...i\\
-0.40996... + 0.82303...i& -0.32227... + 1.89394...i 
\end{pmatrix},\\
Z^\sigma &:= \begin{pmatrix}
-0.14337... + 1.54762...i& 1.99999... - 0.64475...i\\
 2.00000... - 0.64475...i& 0.14337... + 1.54762...i
\end{pmatrix}.
\end{align*}

We compute the Igusa--Clebsch invariants $I_2, I_4, I_6$ and $I_{10}$
to 200 decimal digits of precision using $Z$ and $Z^\sigma$, and
identify them as elements in $F$ (due to
Lemma~\ref{lem:shimura-decomposition}).  In the weighted projective
space $\Proj^2_{(1:2:3:5)}$, this gives the point
\begin{align*}
&(I_2:I_4:I_6:I_{10})=\\
&\quad\left(1: \frac{-21504b + 81889}{5973136}: \frac{-1241984b + 3114075}{1122949568}: 
\frac{1564843b + 21688699}{1362467130944816}\right),
\end{align*}

where $b = \sqrt{53}$. By using Mestre's algorithm~\cite{mestre} which
is implemented in \texttt{Magma}, we obtain a curve with above
invariants. We reduce this curve using the algorithm
in~\cite{bouyer-streng13} implemented in \texttt{Sage} \cite{sage} to
get the curve
\begin{align*}
C_f':\, y^2 &= (-6w + 25)x^6 + (-60w + 246)x^5 + (-242w + 1017)x^4\\
&\qquad + (-534w + 2160)x^3 + (-626w + 2688)x^2\\
&\qquad + (-440w + 1724)x - 127w + 567.
\end{align*}
We have used floating point calculations to get the equation of the
curve $C_f'$, but now we can directly check that the Frobenius data of
its Jacobian matches that of the Hilbert modular form, up to quadratic
twist.

\begin{rem}
We computed the curve $C'_f$ by using the normalized period matrix
$Z$. We could have instead applied the Jacobian nullwerte
method~\cite{gonzalez-guardia02, guardia02} to the periods matrices
$\varOmega_{A_f}$ and $\varOmega_{A_f^\sigma}$. This has the advantage
of producing curves with small coefficients, needing no further
reduction.
\end{rem}

\begin{rem} 
For the other Hilbert modular forms in Case I, we obtained the
corresponding abelian surfaces using Method 1. The only exception is
$D = 61$, where the abelian surface has RM by $\Z[\sqrt{3}]$ and is
naturally $(1,2)$-polarized, and is therefore not principally
polarizable by \cite[Corollary~2.12 and
  Proposition~3.11]{gonzalez-guardia-rotger05}; it is not treated in
this paper.
\end{rem}

\subsubsection{\bf Method 2} \label{subsubsec:method2}

\begin{table}[h]
\caption{The first few Hecke eigenvalues of a base change newform of
  level $(1)$ and weight $2$ over $\Q(\sqrt{53})$. Here $e =
  \sqrt{2}$.}
\label{table:eigenform-sqrt53}
\begin{tabular}{ >{$}c<{$}   >{$}r<{$}   >{$}r<{$}   >{$}r<{$}  >{$}r<{$} }
\toprule
\mathrm{N}\gp&\gp&a_{\gp}(f)& s_{\gp}(f)& t_{\gp}(f)\\
\midrule
4 & 2 & e + 1 & 2 & 7 \\
7 & -w - 2 & -e - 2 & -4 & 16\\
7 & -w + 3 & -e - 2 & -4 & 16\\
9 & 3 & -3e + 1 &  2& 1\\ 
11 & w - 2 & 3e & 0 & 4\\ 
11 & w + 1 & 3e & 0 & 4\\ 
13 & w - 1 & -2e + 1 &  2 & 19\\ 
13 & -w & -2e + 1 &  2 & 19\\ 
17 & -w - 5 & -3 & -6 & 43\\ 
17 & w - 6 & -3 & -6 & 43\\ 
25 & 5 & 2e + 4 &  8 & 58\\ 
29 & -w - 6 & 3e - 3 & -6 & 49\\ 
29 & w - 7 & 3e - 3 & -6 & 49\\ 
\bottomrule
\end{tabular}
\end{table}

An equation for the Hilbert modular surface $Y_{-}(8)$ is given
in~\cite{elkies-kumar12} (see \ref{subsec:hms} for a quick review of
the results we need here). As a double-cover of $\mathbf{P}^2_{r, s}$,
it is given by
\[
z^2 = 2(16rs^2+32r^2s-40rs-s+16r^3+24r^2+12r+2).
\] 
It is a rational surface (even over $\Q$) and therefore the rational
points are dense. In particular, there is an abundance of rational
points of small height. The Igusa--Clebsch invariants $(I_2: I_4: I_6:
I_{10}) \in \Proj^2_{(1:2:3:5)}$ are given by
\[
\left( -\frac{24B_1}{A_1},-12A,\frac{96AB_1-36A_1B}{A_1},-4A_1B_2 \right),
\]
where
\begin{align*}
A_1 &= 2rs^2, \\
A &= -(9rs+4r^2+4r+1)/3, \\
B_1 &= (rs^2(3s+8r-2))/3, \\
B &= -(54r^2s+81rs-16r^3-24r^2-12r-2)/27, \\
B_2 &=  r^2.
\end{align*}

Recall that we expect to find a point of $Y_{-}(8)$ over $F =
\Q(\sqrt{53})$, corresponding to the principally polarized abelian
surface $A$ which should match the Hilbert modular form $f$.  We first
make a list of all $F$-rational points of height $\le 200$ on the
Hilbert modular surface. Next, for each of these rational points, we
try to construct the corresponding genus $2$ curve $C$ over $F$, whose
Jacobian corresponds to the moduli point $(r,s)$ we have chosen, and
check whether the characteristic polynomial of Frobenius on its first
\'etale cohomology group matches up the polynomial $Q_\gp(T)$ giving
the corresponding Euler factor of surface $A_f$ attached to the
Hilbert modular form. If a candidate point $(r,s)$ passes this test
for say the first $50$ primes (ordered by norm) of $F$ of good
reduction for $f$ and $A = J(C)$, we can be reasonably convinced that
it is the correct curve, and then try to prove that $A$ is associated
to $f$.

There are two subtleties in the search. First, since the Hilbert
modular surface $Y_{-}(D')$ is only a coarse moduli space, the point
$(r,s)$ is not enough to recover the curve up to $F$-isomorphism. The
Igusa--Clebsch invariants are rational functions in $r$ and $s$, and
they are only enough to pin down $C$ up to quadratic twist. Therefore,
when we match the quartic $L$-factors $L_\gp(A,T)$ and $Q_\gp(T)$, we
need to allow for
\[
L_\gp(A,\pm T) = Q_\gp(T)
\]
rather than just the plus sign. Second, the Igusa--Clebsch invariants
do not always allow us to define $C$ over the base field $F$; there is
often a Brauer obstruction. Even when $C$ is definable over $F$ (which
is the case we are interested in), it can be computationally expensive
to do so. Therefore, it is convenient to speed up the process of
testing compatibility with $f$ by first reducing $(I_2, I_4, I_6,
I_{10})$ modulo $\mathfrak{p}$ (assuming good reduction) and then
producing a curve $D_\gp$ over $\F_q$ from these reduced invariants,
where $q = \mathrm{N}\mathfrak{p}$. If $C$ exists over $F$, then its
reduction $C_\gp$ will be the same as $D_\gp$ up to quadratic
twist. The advantage is that the Brauer obstruction vanishes over the
finite field $\F_q$, making it very easy to check compatibility at
$\gp$.

In this particular example, a search of $Y_{-}(8)$ for all points of
height $\le 200$ using~\cite{doyle-krumm13} (implemented in
\texttt{Sage}) gives the parameters
\[
r =-\frac{24+10w }{11^2},\,\, s =\frac{136-24w}{11^2},
\]
and the Igusa--Clebsch invariants
 \begin{align*}
I_2&= 208+88w,\\
I_4&= -1660-588w,\\
I_6&= - 428792-135456w,\\
I_{10}&= 643072+204800w.
\end{align*}
This leads to the same curve $C_f'$ as above.

By further reducing the curve we obtained by either of Methods 1 or 2,
we get the following.

\begin{thm}\label{thm:ex-sqrt53} 
Let $C = C_f: y^2+Q(x)y=P(x)$ be the curve over $F$, where
\begin{align*}
P &:= -4x^6 + (w - 17)x^5 + (12w - 27)x^4 + (5w - 122)x^3\\
&\qquad + (45w - 25)x^2 + (-9w - 137)x + 14w + 9, \\
Q &:= wx^3 + wx^2 + w + 1.
\end{align*}
Then
\begin{enumerate}
\item[(a)] The discriminant of this curve is
  $\Delta_C=-\epsilon^{7}$. Thus $C$ has everywhere good reduction.
\item[(b)] The surface $A:=J(C)$ is modular and corresponds to the
  unique Hecke constituent $[f]$ in $S_2(1)$, the space of Hilbert
  cusp forms of weight $2$ and level $(1)$ over $F=\Q(\sqrt{53})$.
\end{enumerate}
\end{thm}

\begin{proof} 
A direct calculation shows that $\Delta_C=-\epsilon^{7}$.  By
construction, $A$ has real multiplication by $\CO_f=\Z[\sqrt{2}]$,
where $7$ is split.  Let $\lambda$ be one of the primes above $7$, and
consider the $\lambda$-adic representation
\[
\rho = \rho_{A, \lambda}:\,\Gal(\Qbar/F)\to
\GL_2(K_{f,\lambda})\simeq\GL_2(\Q_7),
\] 
and its reduction $\overbar{\rho}$ modulo $\lambda$. We will show that
$\rho$ is modular by using~\cite[Theorem A]{skinner-wiles99}. For
this, it suffices to show that $\overbar{\rho}$ is reducible or,
equivalently, that $A$ has a $7$-torsion point defined over $F$.  By
definition, we have
\[
A(F) \simeq \mathrm{Pic}^0(C)(F).
\]
So it is enough to find a degree zero divisor $D$ defined over $F$
such that $7D$ is principal.  To this end, we consider the field $L =
F(\alpha)$, where $\alpha$ is a root of the polynomial $x^2 - wx +
3$. Let $\sigma'\in \Gal(L/F)$ be the non-trivial involution.  Then,
the point $P = \left(\alpha, (-6w - 12)\alpha + 2w + 18\right)\in
C(L)$, and the divisor $D := P + \sigma'(P) - 2\infty$ belongs to
$\mathrm{Pic}^0(C)(F)$.  An easy calculation shows that $7D \sim
(0)$. Hence, $D$ corresponds to a point of order $7$ in $A(F)$.

Since $S_2(1)$ has dimension $2$ and is spanned by $[f]$, $A$ must
correspond to this Hilbert newform.
\end{proof}

\begin{rem}
Both $C$ and $A$ have everywhere good reduction. However, this is not
true in some of the other examples. Indeed, it can happen that a curve
$C$ has bad reduction at a prime $\gp$ while $\mathrm{Jac}(C)$ does
not. (See the example of Theorem~\ref{thm:ex-sqrt929}.)
\end{rem}

\begin{rem}
The modularity of the abelian surface $A = \mathrm{Jac}(C)$ we found
means that it is isogenous to the surface $A_f$ obtained from the
Eichler--Shimura construction over $\Q$. Since $A_f$ is a
$\Q$-surface, so is $A$. In fact, the proof of the reducibility of
$\overbar{\rho}_{A,\lambda}$ implies that $A$ and its Galois conjugate are
related by a $7$-isogeny.
\end{rem}

\subsection{Case II} 
The following result explains the connection between
Conjectures~\ref{conj:eichler-shimura} and~\ref{conj:brumer-kramer}.

\begin{prop}\label{prop:bk-implies-es} 
Assume that Conjecture~\ref{conj:brumer-kramer} is true. Let $F$ be a
real quadratic field. Let $f$ be a Hilbert newform of weight $2$ and
level $\gN$ over $F$, which satisfies the hypotheses of Case II.  Then
$f$ satisfies Conjecture~\ref{conj:eichler-shimura}.
\end{prop}

\begin{proof}
Since $f$ is a non-base change,~\cite[Main
  Theorem]{johnson-leung-roberts12} implies that there is a
paramodular Siegel newform $g$ of genus $2$, level $ND^2$ and weight
$2$ attached to $f$, where $N=\mathrm{N}_{F/\Q}(\gN)$. Moreover, since
$\Gal(F/\Q)$ preserves $\{f, f^\tau\}$, we must have
\[
a_{\gp^\sigma}(f) =a_{\gp}(f)^\tau
\]
for all primes $\gp \subseteq \CO_F$. Therefore, the Hecke eigenvalues
of the form $g$ are integers. So by
Conjecture~\ref{conj:brumer-kramer}, there is an abelian surface $B_g$
defined over $\Q$ with $\End_{\Q}(B_g)=\Z$ such that $L(B_g, s)=L(g,
s)$. Let $A_f$ be the base change of $B_g$ to $F$. Then, by
construction, we have
\[
L(A_f, s) = L(f, s) L(f^\tau, s).
\] 
Hence, $A_f$ satisfies Conjecture~\ref{conj:eichler-shimura}.
\end{proof}

\begin{rem}
Assume Conjecture~\ref{conj:brumer-kramer}. By
Proposition~\ref{prop:bk-implies-es}, if $A_f$ is an abelian surface
attached to a Hilbert newform $f$ satisfying Case II, then $A_f$ is
the base change to $F$ of some surface $B$ defined over $\Q$, which
acquires extra endomorphisms. Therefore, we know that the
Igusa--Clebsch invariants of $A_f$ are in $\Q$, and we can use this
fact in looking for $A_f$.
\end{rem}

The first real quadratic field of narrow class number $1$ where there
is a form $f$ of level $(1)$ and weight $2$, which satisfies Case II,
is $F=\Q(\sqrt{193})$ (see Table~\ref{table:eigenform-sqrt193}). The
coefficients of $f$ generate the ring of integers
$\CO_f:=\Z[\frac{1+\sqrt{17}}{2}]$ of the field $K_f=\Q(\sqrt{17})$.

\begin{table}[h]
\caption{The first few Hecke eigenvalues of a non-base change
  newform of level $(1)$ and weight $2$ over $\Q(\sqrt{193})$. 
  Here $e = (1 + \sqrt{17})/2$.}
\label{table:eigenform-sqrt193}
\begin{tabular}{ >{$}c<{$}   >{$}r<{$}   >{$}r<{$}   >{$}r<{$}  >{$}r<{$} }
\toprule
\mathrm{N}\gp&\gp&a_{\gp}(f)& s_{\gp}(f) & t_{\gp}(f)\\
\midrule
 2& 9w - 67& e& 1 & 0 \\
2& 9w + 58& -e + 1& 1 & 0 \\
3& -2w + 15& e&  1 & 2 \\
3& 2w + 13& -e + 1& 1 & 2 \\
7& -186w - 1199& -e + 2& 3 & 12 \\
7& 186w - 1385& e + 1&  3 & 12 \\
23& 38w - 283& -e - 6& - 13 & 84 \\
23& -38w - 245& e - 7& - 13 & 84 \\
25& 5& 1& 2 & 51 \\
31& -16w - 103& e - 3& - 5 & 64 \\
31& -16w + 119& -e - 2& - 5 & 64 \\
43& 4w + 25 &  e+4& 9 & 102\\
43& -4w + 29& -e+5& 9 & 102\\
\bottomrule
\end{tabular}
\end{table}

\begin{thm}\label{thm:ex-sqrt193} 
Let $C: y^2+Q(x)y=P(x)$ be the curve over $F$, where
\begin{align*}
P(x)&:=2x^6 + (-2w + 7)x^5 + (-5w + 47)x^4 + (-12w +
85)x^3\\ &{}\qquad + (-13w + 97)x^2 + (-8w + 56)x - 2w +
1,\\ Q(x)&:=-x-w.
\end{align*}
Then
\begin{enumerate}
\item[(a)] The discriminant $\Delta_C=-1$, hence $C$ has everywhere
  good reduction.
\item[(b)] The surface $J(C)$ is modular and corresponds to the form
  $f$ listed in Table~\ref{table:eigenform-sqrt193}.
\end{enumerate}
\end{thm}

\begin{rem}
A theorem of Stroeker~\cite{stroeker83} implies\footnote{Stroeker's
  result is stated for imaginary quadratic fields. Elkies
  \cite{elkies-db} remarks that the argument implies the statement
  above for real quadratic fields.} that if $E$ is an elliptic curve
defined over a real quadratic field $F$ having good reduction
everywhere, then $\Delta_E\notin \{-1,1\}$. However, this fails for
curves of genus $2$, by the above example.
\end{rem}

\begin{proof} 
We show that $\Delta_C=-1$ as before, which implies that $C$ and
$J(C)$ both have everywhere good reduction.  However, it is important
to observe that we located the curve based on our heuristics which
rely on Conjectures~\ref{conj:eichler-shimura}
and~\ref{conj:brumer-kramer}. Indeed, let $S_2(1)$ be the space of
Hilbert cuspforms of level $(1)$ and weight $2$ over
$F=\Q(\sqrt{193})$.  Then $S_2(1)$ has dimension $9$, and decomposes
into two Hecke constituents of dimension $2$ and $7$ respectively. The
form $f$ in Table~\ref{table:eigenform-sqrt193} is an eigenvector in
the $2$-dimensional constituent, and it is a non-base change whose
Hecke constituent is Galois stable. So we can look for our surface
$A_f$ with the help of Proposition~\ref{prop:bk-implies-es}.

To find the curve $C$, we proceed as in Sect.~\ref{subsubsec:method2},
using the results from \cite{elkies-kumar12}. The surface $Y_{-}(17)$
is a double-cover of the (weighted) projective space $\mathbf{P}_{g,
  h}^2/\Q$ given by
\begin{align*}
z^2 &= -256\,h^3 + (192\,g^2+464\,g+185)\,h^2 \\
    &{} \qquad-2\,(2\,g+1)\,(12\,g^3-65\,g^2-54\,g-9)\,h +
    (g+1)^4\,(2\,g+1)^2.
\end{align*}

A search for $\Q$-rational points of low height on this surface yields
the following parameters, Igusa--Clebsch and $G_2$ invariants:
\begin{align*}
g& = 0, h = -1/4,\\
I_2 &= 40,\, I_4= -56,\, I_6 = -669,\, I_{10} = -4,\\
j_1 &= -3200000,\, j_2 = -208000,\, j_3 = -16400.
\end{align*}
Over $\Q$, this gives the curve
\[
C':\, y^2=-8x^6 + 220x^5 - 44x^4 - 14828x^3 - 4661x^2 - 21016x +
10028.
\]
After finding a suitable twist and reducing the Weierstrass equation,
we get the curve $C$ displayed in the statement of the theorem.

\medskip

To prove modularity, we note that $3$ is inert in $K_f=\Q(\sqrt{17})$,
and consider the $3$-adic representation attached to $A$,

\begin{align*}
\rho_{A,3}:\,\Gal(\Qbar/F)&\to \GL_2(K_{f,(3)})\simeq\GL_2(\Q_9).
\end{align*}
By computing the orders of Frobenius for the first few primes, we see
that the $\bmod\, 3$ representation
  \begin{align*}
    \overbar{\rho}_{A,3}:\,\Gal(\Qbar/F)&\to\GL_2(\F_9)
  \end{align*}
is surjective, and absolutely irreducible. Hence $\rho_{A, 3}$ is also
absolutely irreducible.  Since $3$ and $5$ are unramified in the
quadratic field $F$, the ramification indices of $\overbar{\rho}_{A,3}$ at
the primes of $F$ above them are odd. Also, since $\overbar{\rho}_{A, 3}$
is unramified at $(5)$, the image of the inertia group at $I_{(5)}$ at
$5$ in $\mathrm{GL}_2(\F_9)$ is trivial. In particular, the image of
$I_{(5)}$ has odd order and lies in $\mathrm{SL}_2(\F_9)$. By studying
the Tate module of $A\times_F F(\zeta_3)$, we also see that
$\overbar{\rho}|_{G_{F(\zeta_3)}}$ is absolutely irreducible.  Therefore,
$\overbar{\rho}_{A, 3}$ is modular by \cite[Theorem 3.2 and Proposition
  3.4]{ellenberg05}. We then apply \cite[Theorem 1.1 in
  Erratum]{gee06, gee09} to conclude that $\rho_{A, 3}$ is modular.
So, $A$ is modular and corresponds to the unique newform $f \in
S_2(1)$ with coefficients in $\CO_f=\Z[\frac{1+\sqrt{17}}{2}]$.
\end{proof}

\begin{cor}\label{cor:bk-evidence}
Let $B$ be the Jacobian of the curve $C'/\Q$ in the proof of
Theorem~\ref{thm:ex-sqrt193}. Then $B$ is paramodular of level
$193^2$.
\end{cor}

\begin{rem} 
In~\cite{brumer-kramer12}, the authors remarked that Conjecture 1.4 in
their paper should be verifiable by current technology for paramodular
abelian surfaces $B$ over $\Q$ with $\End_{\Qbar}(B)\supsetneq
\Z$. The majority of the surfaces we found fall in Case II (see
Sect.~\ref{sec:data}), and provide such evidence by
Corollary~\ref{cor:bk-evidence}.
\end{rem}

In contrast to the curves in Theorems~\ref{thm:ex-sqrt53}
and~\ref{thm:ex-sqrt193}, we found a few curves whose Jacobians had
everywhere good reduction while the curves themselves did not. We now
discuss one such example, for the field $F = \Q(\sqrt{929})$, with
Hecke eigenvalues in $\Q(\sqrt{13})$.

\begin{table}[h]
\caption{The first few Hecke eigenvalues of a non-base change
  newform of level $(1)$ and weight $2$ over $\Q(\sqrt{929})$. 
  Here $e = (1 + \sqrt{13})/2$.}
\label{table:eigenform-sqrt929}
\begin{tabular}{ >{$}c<{$}   >{$}r<{$}   >{$}r<{$}   >{$}r<{$}  >{$}r<{$} }
\toprule
\mathrm{N}\gp&\gp&a_{\gp}(f)& s_{\gp}(f) & t_{\gp}(f) \\
\midrule
2 & 561w - 8830 & -e + 1 & 1 & 1  \\ 
2 & 561w + 8269 & e&  1 & 1 \\ 
5 & -4w - 59 & -e + 1&  1& 7 \\ 
5 & 4w - 63 & e&  1 & 7 \\ 
9 & 3 & 3& 6&  27 \\ 
11 & -8342w + 131301 & 2e - 3& -4 & 13  \\ 
11 & 8342w + 122959 & -2e - 1& -4 & 13  \\ 
19 & -50w - 737 & e - 2& -3 & 37 \\ 
19 & 50w - 787 & -e - 1&  -3 & 37 \\ 
23 & -42832w + 674165 & 4e - 4& -4 & - 2 \\ 
23 & 42832w + 631333 & -4e&  -4 & - 2 \\ 
29 & -2w + 31 & -2e + 6& 10 & 70 \\ 
29 & 2w + 29 & 2e + 4& 10 & 70 \\ 
\bottomrule
\end{tabular}
\end{table}

\begin{thm}\label{thm:ex-sqrt929} 
Let $C: y^2+Q(x)y=P(x)$ be the curve over $F$, where
\begin{align*}
P(x)&:=23x^6 + (90w - 45)x^5 + 33601x^4 + (28707w -
14354)x^3\\ &\quad\qquad{}\,\,\, + 3192149x^2 + (811953w - 405977)x +
19904990,\\ Q(x)&:=x^3 + x + 1.
\end{align*}
Then
\begin{enumerate}
\item[(a)] The discriminant $\Delta_C=3^{22}$, hence $C$ has bad
  reduction at $(3)$.
\item[(b)] The surface $A:=J(C)$ has everywhere good reduction. It is
  modular and corresponds to the form $f$ listed in
  Table~\ref{table:eigenform-sqrt929}.
\end{enumerate}
\end{thm}

\begin{proof} 
The curve $C$ is a global minimal model for the base change to $F$ of
the curve $C'/\Q$ given by
\begin{align*}
C': y^2 &= 93x^6 - 14688x^5 + 549594x^4 + 2268918x^3 + 2259369x^2 - 1488402x\\
&{}\qquad + 13059345.
\end{align*}
We compute the reduction $\widetilde{C'}$ of $C'$ at $3$ by
combining~\cite[Theorem 1 and Proposition 2]{liu93}, and Liu's
algorithm implemented in \texttt{Sage}. This returns the type
$\mathrm{(V),\,[I_{0}-I_{0}-1]}$.  So the reduction $\widetilde{A'}$
of the Jacobian $A'$ of $C'$ is a product of two elliptic curves whose
$j$-invariants are $j_1=j_2=0$ (\cite[Proposition 2, (v)]{liu93}).
This implies that $A'$ has non-ordinary good reduction at $(3)$; and
so does $A$ since $3$ is inert in $F$. (Note that this is consistent
with the fact that $a_{(3)}(f)=3$.)  Since $3$ is the only prime
dividing $\Delta_C$, we see that $A$ has everywhere good reduction.

To prove modularity, we recall that by construction $A$ has real
multiplication by $\CO_f=\Z[\frac{1+\sqrt{13}}{2}]$, where $3$
splits. We choose a prime $\lambda$ above $3$, and consider the
$\lambda$-adic representation 
\[
\rho_{A, \lambda}:\,\Gal(\Qbar/F)\to
\GL_2(K_{f,\lambda})\simeq\GL_2(\Q_3)
\]
and its reduction $\overbar{\rho}_{A,\lambda}$ modulo $\lambda$. By
computing the first few Frobenii, we see that $\overbar{\rho}_{A,
  \lambda}$ is surjective, hence irreducible. Since $\GL_2(\F_3)$ is
solvable, $\overbar{\rho}$ is modular by
Langlands--Tunnell~\cite[Chap. I]{langlands80} and~\cite{tunnell81}.
By looking at the Tate module of $A\times_F F(\zeta_3)$, we also see
that $\rho_{A,\lambda}$ is not induced from $F(\zeta_3)$. So, we
conclude that $\rho_{A, \lambda}$ is modular by \cite[Theorem 1.1 in
  Erratum]{gee06, gee09}.
\end{proof}

\begin{rem}
The example in Theorem~\ref{thm:ex-sqrt929} and other similar ones in
Table~\ref{table:case-II} underscore the difficulty in producing
effective algorithms for principally polarized abelian surfaces with
good reduction outside a (finite) prescribed set of primes $S$ of
$\CO_F$.  Indeed, let $A$ be such a surface so that $A =
\mathrm{Jac}(C)$, where $C$ is a curve defined over $F$ with good
reduction outside a finite set of primes $T\supseteq S$. Then, the set
$T \setminus S$ is non-empty in general, depends {\it a priori} on
$A$, and is hard to predict. When $A$ has real multiplication by some
quadratic field $K$ and is attached to a modular form $f$, $T\setminus
S$ is contained in the set of non-ordinary primes for $f$, which is
possibly infinite.
\end{rem}

Similar proofs apply for the other Hilbert modular forms in Case II
for which we were able to find matching principally polarized abelian
surfaces. However, there are five examples (listed in Table
\ref{table:unresolved}) for which we were unable as yet to find
matching abelian surfaces. In each case, the Fourier coefficients of
the form indicate that the missing surface would have real
multiplication by the full ring of integers $\CO_{D'}$. So, assuming
the Eichler--Shimura conjecture holds, our difficulties in matching
those forms could be due to one of the following reasons:
\begin{enumerate}
\item[(a)] Our height bound for the rational point search on the
  corresponding Hilbert modular surfaces is too small. We searched for
  parameters $r,s \in \Q$ of height up to $1000$.
\item[(b)] The corresponding abelian surface is not principally
  polarized. Note that the criteria given in
  \cite[Proposition~3.11]{gonzalez-guardia-rotger05} to convert an
  arbitrary polarization to a principal polarization fail for each of
  the missing discriminants $D'$. For $(D,D') = (677,85)$, the field
  $\Q(\sqrt{D'})$ has class number $2$, whereas for the other
  examples, there is no unit of negative norm.
\end{enumerate}
There is also the possibility, since the models in
\cite{elkies-kumar12} are birational to $Y_{-}(D')$ (rather than
isomorphic), that we might have missed some curves or points in our
search. However, this is unlikely to be the case, as the extra points
should correspond to abelian surfaces with extra endomorphisms.

\subsection{Case III} 

This is by far the trickiest case, since the Igusa--Clebsch invariants
(and therefore $r,s$) are not in $\Q$. This leads to a much slower
search for $F$-points on $Y_{-}(D')$, compared to searching for
$\Q$-points. We searched for points of height up to $400$ using the
enumeration of points of small height developed in
\cite{doyle-krumm13} (implemented in \texttt{Sage}), but were unable
to find either of the two examples predicted by the Eichler--Shimura
conjecture, corresponding to the Hilbert modular forms of level $1$
and weight $2$ over $\Q(\sqrt{353})$ and $\Q(\sqrt{421})$, both with
Fourier coefficients in $\Q(\sqrt{5})$. In addition to the reduced
search height bound, another complicating factor is the fundamental
unit of $F$, which might be quite large. In Case II, the discriminant
of the genus $2$ curve differed from $I_{10}(r,s)$ by only a few small
(rational) primes. However, in Case III, one has to take into account
the fact that a power of the fundamental unit might also appear in the
discriminant. On the other hand, principal polarizability is not an
obstruction, as $\Q(\sqrt{5})$ has a fundamental unit of negative
norm.

We hope to address the missing examples using different techniques in
future work.

\begin{table}[h] 
\caption{Unresolved Cases}
\label{table:unresolved}
\begin{tabular}{cc}
\toprule
Case & List of $(D,D')$ \\
\midrule
II & $(433,12), (613,21), (677,85), (821,44), (853,21)$ \\
III & $ (353,5), (421,5)$ \\
\bottomrule
\end{tabular}
\end{table}

\section{\bf \large The data}\label{sec:data}

In tables~\ref{table:case-I} and \ref{table:case-II} we list genus $2$
curves $y^2 + Q(x) y = P(x)$ matching the data. We always set $b =
\sqrt{D}$ and $w = (b+1)/2$. We suppress $Q(x)$ when it is $0$. (We
recall that each of the curves listed has a modular Jacobian. In Case
I, this is true as the Jacobian is a $\Q$-surface. While in Case II,
we prove the modularity by the same technique as above.)

\clearpage
\begin{small}
\begin{table}[t!]
\caption{Case I examples}
\label{table:case-I}
\begin{tabular}{ >{$}c<{$}   >{$}c<{$}   >{$}l<{$}   }
D & D' & \hspace{20mm} \textrm{Hyperelliptic polynomials} \\
\midrule
 & & Q = wx^3 + wx^2 + w + 1 \\
53 & 8 & P = -4x^6 + (w - 17)x^5 + (12w - 27)x^4 + (5w - 122)x^3  \\ 
& & \qquad \quad + (45w - 25)x^2 + (-9w - 137)x + 14w + 9  \\ 
\cmidrule{3-3}
& & Q = x^3 + x + 1 \\
73 & 5 & P=  (w - 5)x^6 + (3w - 14)x^5 + (3w - 19)x^4  \\
& & \qquad \quad + (4w - 3)x^3 -(3w + 16)x^2 + (3w + 11)x - (w + 4) \\
\cmidrule{3-3}
& & Q =  w(x^3 + 1)  \\
421 & 5 & P=  -2(4414w + 43089)x^6 + (31147w+303963)x^5 \\
& & \qquad \quad -10(4522w+44133)x^4 + 2(17290w+168687)x^3 \\
& & \qquad \quad -18(816w+7967)x^2 + 27(122w+1189)x -(304w+3003) \\
\bottomrule
\end{tabular}
\vspace{6mm}
\end{table}
\end{small}

\begin{small}
\begin{table}[b!]
\caption{Case II examples}
\label{table:case-II}
\begin{tabular}{ >{$}c<{$}   >{$}c<{$}   >{$}l<{$}   }
D & D' & \hspace{20mm} \textrm{Hyperelliptic polynomials} \\
\midrule
& & Q =  -x-w \\
193 & 17 & P=  2x^6 + (-2w + 7)x^5 + (-5w + 47)x^4 + (-12w +85)x^3 \\
& & \qquad \quad + (-13w + 97)x^2 + (-8w + 56)x - 2w + 1 \\
\cmidrule{3-3}
233 & 17 & Q = x+1 \\
& & P=  -2x^6 - (2w - 1)x^5 - 45x^4 - 4(2w -1)x^3- 31x^2 + (w - 1)x + 9\\
\cmidrule{3-3}
277 & 29 & Q =  -1 \\
&  & P=  -24 x^6 + 31 b x^5 - 4615 x^4 + 1321 b x^3 +58837 x^2 + 5039 b x    - 49745\\
\cmidrule{3-3}
349 & 21 & Q = x^3 \\
 &  & P=  -2x^6 + 4bx^5 - 1328x^4 + 673bx^3 - 66879x^2 + 10145bx - 223536 \\
\cmidrule{3-3}
373 & 93 & P= (265b - 5118)(8x-b)(8960bx^5 - 2020471x^4 + 488608bx^3 \\
& & \qquad \quad  - 22037369x^2 + 1332394bx - 12019522)  \\
\cmidrule{3-3}
389 & 8 & Q =  x^3 + x^2 + x + 1 \\
 &  & P=  -wx^5 + 159x^4 - (138w - 68)x^3 + 6429x^2 - (1619w - 809)x + 16260 \\
\cmidrule{3-3}
397 & 24 & P=  -601 x^6 + 748 b x^5 - 154001 x^4 + 42596 b x^3 - 2631127 x^2 \\
& & \qquad \quad + 218342 b x - 2997270 \\
\cmidrule{3-3}
& & Q = x^3 + x^2 + 1 \\
409 & 13 & P=  -2x^6 + (-3w + 1)x^5 - 219x^4 + (-83w + 41)x^3 - 1806x^2 \\
& & \qquad \quad + (-204w + 102)x - 977 \\
\cmidrule{3-3}
461 & 29 & Q = x^3 \\
 & & P=  -32x^6 -34bx^5 - 6916x^4 -1605bx^3 - 94873x^2 - 6335bx - 78584  \\
\cmidrule{3-3}
& & Q = x^3 + x^2 + x + 1 \\
677 & 13 & P=  -12x^6 + (61w - 31)x^5 - 22335x^4 + (25770w - 12886)x^3 \\
 & & \qquad \quad  -2830998x^2 + (980087w - 490044)x - 23929668 \\
\bottomrule
\end{tabular}
\end{table}
\end{small}

\clearpage
\addtocounter{table}{-1}
\begin{small}
\begin{table*}[h]
\caption{continued}
\begin{tabular}{ >{$}c<{$}   >{$}c<{$}   >{$}l<{$}   }
D & D' & \hspace{20mm} \textrm{Hyperelliptic polynomials} \\
\midrule
677 & 29 & P= -4453 x^6 - 5786 b x^5 - 2120768 x^4 - 612392 b x^3- 67342400 x^2 \\
& & \qquad \quad  - 5834038 b x - 142573513 \\
\cmidrule{3-3}
709 & 5 & P= 2x^6 + 2 b x^5 + 560 x^4 + 114 b x^3 + 9040 x^2 + 530 b x + 9058  \\
\cmidrule{3-3}
797 & 8 & P= 1856x^6 -3784bx^5 + 2561907x^4 -1160668bx^3+ 235735797x^2  \\
& & \qquad \quad -32038746bx + 1445987770 \\
\cmidrule{3-3}
& & Q =  x^3 + x^2 + x + 1  \\
797 & 29 & P= x^6 + (3w - 2)x^5 + 594x^4 + (314w - 158)x^3 + 18483x^2   \\
& & \qquad \quad + (2897w - 1449)x + 37491\\
\cmidrule{3-3}
& & Q = x^3 + x + 1 \\
809 & 5 & P= -134 x^6 - (146 w - 73) x^5 - 13427 x^4 - (3255w - 1627) x^3  \\
& & \qquad \quad -89746 x^2 -(6523w - 3261) x - 39941  \\
\cmidrule{3-3}
& & Q = x^3 + x + 1 \\
929 & 13 & P = 23x^6 + (90w - 45)x^5 + 33601x^4 + (28707w -14354)x^3 \\
& & \qquad \quad + 3192149x^2 + (811953w - 405977)x + 19904990 \\
\cmidrule{3-3}
997 & 13 & Q = x^3 \\
 &  & P = x^6 + 3bx^5 + 2989x^4 + 1592bx^3 + 475212x^2 + 75831bx + 5023486 \\
\bottomrule
\end{tabular}
\end{table*}
\end{small}

\section{\bf \large Appendix}

In Table~\ref{table:hmfdata} below we list Hilbert modular form data
for all the examples considered in this paper.

\begin{table}[h!] 
\caption{Hecke eigenvalues for the Hilbert modular forms in this
  paper}
\label{table:hmfdata}
\begin{center}
\scriptsize
\begin{tabular}{@{}crrccrrccrr@{}} \toprule
\multicolumn{3}{c}{$D=53,\,\,D'=8$}&\phantom{abc}&\multicolumn{3}{c}{$D=61,\,\,D'=12$}&\phantom{abc}&\multicolumn{3}{c}{$D=73,\,\,D'=5$}\\\cmidrule{1-3}\cmidrule{5-7}\cmidrule{9-11}
$\mathbf{N}\mathfrak{p}$&$\mathfrak{p}$&$a_{\mathfrak{p}}(f)$&&$\mathbf{N}\mathfrak{p}$&$\mathfrak{p}$&$a_{\mathfrak{p}}(f)$&&$\mathbf{N}\mathfrak{p}$&$\mathfrak{p}$&$a_{\mathfrak{p}}(f)$\\\midrule
4&$2$&$e + 1$&&3&$-w - 3$&$e - 1$&&2&$-w - 4$&$-e$\\
7&$-w - 2$&$-e - 2$&&3&$-w + 4$&$e - 1$&&2&$w - 5$&$-e$\\
7&$-w + 3$&$-e - 2$&&4&$2$&$e$&&3&$-4w - 15$&$-e + 1$\\
9&$3$&$-3e + 1$&&5&$w - 5$&$-e$&&3&$-4w + 19$&$-e + 1$\\
11&$w - 2$&$3e$&&5&$-w - 4$&$-e$&&19&$6w - 29$&$4e - 1$\\
11&$w + 1$&$3e$&&13&$-w - 1$&$3$&&19&$-6w - 23$&$4e - 1$\\
13&$w - 1$&$-2e + 1$&&13&$w - 2$&$3$&&23&$14w - 67$&$-3e + 4$\\
13&$-w$&$-2e + 1$&&19&$-3w - 11$&$-e + 3$&&23&$-14w - 53$&$-3e + 4$\\
17&$-w - 5$&$-3$&&19&$3w - 14$&$-e + 3$&&25&$5$&$-e + 1$\\
17&$w - 6$&$-3$&&41&$w - 8$&$-2e - 3$&&37&$-2w - 5$&$5$\\
25&$5$&$2e + 4$&&41&$-w - 7$&$-2e - 3$&&37&$2w - 7$&$5$\\
29&$-w - 6$&$3e - 3$&&47&$-3w - 8$&$4e + 6$&&41&$30w - 143$&$2e + 4$\\\bottomrule
\end{tabular}
\end{center}
\end{table}

\addtocounter{table}{-1}

\begin{table*}[t!] 
\caption{continued}
\begin{center}
\scriptsize
\begin{tabular}{@{}crrccrrccrr@{}} \toprule
\multicolumn{3}{c}{$D=193,\,\,D'=17$}&\phantom{abc}&\multicolumn{3}{c}{$D=233,\,\,D'=17$}&\phantom{abc}&\multicolumn{3}{c}{$D=277,\,\,D'=29$}\\\cmidrule{1-3}\cmidrule{5-7}\cmidrule{9-11}
$\mathbf{N}\mathfrak{p}$&$\mathfrak{p}$&$a_{\mathfrak{p}}(f)$&&$\mathbf{N}\mathfrak{p}$&$\mathfrak{p}$&$a_{\mathfrak{p}}(f)$&&$\mathbf{N}\mathfrak{p}$&$\mathfrak{p}$&$a_{\mathfrak{p}}(f)$\\\midrule
2&$9w + 58$&$-e + 1$&&2&$-w - 7$&$e$&&3&$w + 8$&$-e + 1$\\
2&$9w - 67$&$e$&&2&$-w + 8$&$-e + 1$&&3&$-w + 9$&$e$\\
3&$-2w + 15$&$e$&&7&$-8w + 65$&$e - 1$&&4&$2$&$-2$\\
3&$2w + 13$&$-e + 1$&&7&$8w + 57$&$-e$&&7&$6w - 53$&$-e + 3$\\
7&$186w - 1385$&$e + 1$&&9&$3$&$-2$&&7&$-6w - 47$&$e + 2$\\
7&$-186w - 1199$&$-e + 2$&&13&$38w - 309$&$-e + 3$&&13&$-w - 7$&$-e - 1$\\
23&$-38w + 283$&$-e - 6$&&13&$-38w - 271$&$e + 2$&&13&$w - 8$&$e - 2$\\
23&$-38w - 245$&$e - 7$&&19&$-6w + 49$&$-3e + 3$&&19&$4w + 31$&$-2e + 1$\\
25&$5$&$1$&&19&$6w + 43$&$3e$&&19&$-4w + 35$&$2e - 1$\\
31&$-16w + 119$&$-e - 2$&&23&$2w + 15$&$-e + 2$&&23&$-3w + 26$&$3$\\
31&$-16w - 103$&$e - 3$&&23&$-2w + 17$&$e + 1$&&23&$-3w - 23$&$3$\\
43&$4w + 25$&$e + 4$&&25&$5$&$-3$&&25&$5$&$-3$\\\bottomrule
\multicolumn{11}{c}{}\\
\multicolumn{11}{c}{}\\
\toprule
\multicolumn{3}{c}{$D=349,\,\,D'=21$}&\phantom{abc}&\multicolumn{3}{c}{$D=353,\,\,D'=5$}&\phantom{abc}&\multicolumn{3}{c}{$D=353,\,\,D'=5$}\\\cmidrule{1-3}\cmidrule{5-7}\cmidrule{9-11}
$\mathbf{N}\mathfrak{p}$&$\mathfrak{p}$&$a_{\mathfrak{p}}(f)$&&$\mathbf{N}\mathfrak{p}$&$\mathfrak{p}$&$a_{\mathfrak{p}}(f)$&&$\mathbf{N}\mathfrak{p}$&$\mathfrak{p}$&$a_{\mathfrak{p}}(f)$\\\midrule
3&$-w - 9$&$-e + 1$&&2&$w + 9$&$2e - 1$&&2&$w + 9$&$-e + 1$\\
3&$w - 10$&$e$&&2&$-w + 10$&$-e + 1$&&2&$-w + 10$&$2e - 1$\\
4&$2$&$-2$&&9&$3$&$-2e - 2$&&9&$3$&$-2e - 2$\\
5&$-6w + 59$&$e$&&11&$10w + 89$&$-2e + 2$&&11&$10w + 89$&$2e + 3$\\
5&$-6w - 53$&$-e + 1$&&11&$-10w + 99$&$2e + 3$&&11&$-10w + 99$&$-2e + 2$\\
17&$13w - 128$&$-e + 2$&&17&$66w - 653$&$-4e + 2$&&17&$66w - 653$&$3$\\
17&$13w + 115$&$e + 1$&&17&$-66w - 587$&$3$&&17&$-66w - 587$&$-4e + 2$\\
19&$-5w - 44$&$2e$&&19&$-28w + 277$&$2$&&19&$-28w + 277$&$2e - 3$\\
19&$5w - 49$&$-2e + 2$&&19&$28w + 249$&$2e - 3$&&19&$28w + 249$&$2$\\
23&$-w + 11$&$-2e + 5$&&23&$-8w - 71$&$4e - 2$&&23&$-8w - 71$&$2e + 3$\\
23&$w + 10$&$2e + 3$&&23&$8w - 79$&$2e + 3$&&23&$8w - 79$&$4e - 2$\\
29&$-3w - 26$&$-2e - 1$&&25&$5$&$-3$&&25&$5$&$-3$\\\bottomrule
\multicolumn{11}{c}{}\\
\multicolumn{11}{c}{}\\
\toprule
\multicolumn{3}{c}{$D=373,\,\,D'=93$}&\phantom{abc}&\multicolumn{3}{c}{$D=389,\,\,D'=8$}&\phantom{abc}&\multicolumn{3}{c}{$D=397,\,\,D'=24$}\\\cmidrule{1-3}\cmidrule{5-7}\cmidrule{9-11}
$\mathbf{N}\mathfrak{p}$&$\mathfrak{p}$&$a_{\mathfrak{p}}(f)$&&$\mathbf{N}\mathfrak{p}$&$\mathfrak{p}$&$a_{\mathfrak{p}}(f)$&&$\mathbf{N}\mathfrak{p}$&$\mathfrak{p}$&$a_{\mathfrak{p}}(f)$\\\midrule
3&$w - 10$&$-2$&&4&$2$&$-2$&&3&$2w + 19$&$-e$\\
3&$w + 9$&$-2$&&5&$-3w - 28$&$2e - 1$&&3&$-2w + 21$&$e$\\
4&$2$&$3$&&5&$-3w + 31$&$-2e - 1$&&4&$2$&$-1$\\
7&$-6w - 55$&$-2$&&7&$-w - 9$&$-2e - 1$&&11&$w - 11$&$-e + 2$\\
7&$6w - 61$&$-2$&&7&$w - 10$&$2e - 1$&&11&$-w - 10$&$e + 2$\\
13&$-7w + 71$&$e + 1$&&9&$3$&$-4$&&19&$-11w + 115$&$2e - 2$\\
13&$-7w - 64$&$e + 1$&&11&$2w + 19$&$-2e - 2$&&19&$-11w - 104$&$-2e - 2$\\
17&$-w - 10$&$e - 2$&&11&$-2w + 21$&$2e - 2$&&23&$-3w - 28$&$2$\\
17&$-w + 11$&$e - 2$&&13&$-w - 10$&$2e + 1$&&23&$3w - 31$&$2$\\
25&$5$&$6$&&13&$w - 11$&$-2e + 1$&&25&$5$&$-4$\\
29&$-4w + 41$&$-e - 1$&&17&$-8w - 75$&$2e - 4$&&29&$9w - 94$&$1$\\
29&$-4w - 37$&$-e - 1$&&17&$-8w + 83$&$-2e - 4$&&29&$-9w - 85$&$1$\\\bottomrule
\end{tabular}
\end{center}
\end{table*}

\addtocounter{table}{-1}

\begin{table*}[t!]
\caption{continued}
\begin{center}
\scriptsize
\begin{tabular}{@{}crrccrrccrr@{}} \toprule
\multicolumn{3}{c}{$D=409,\,\,D'=13$}&\phantom{abc}&\multicolumn{3}{c}{$D=421,\,\,D'=5$}&\phantom{abc}&\multicolumn{3}{c}{$D=421,\,\,D'=5$}\\\cmidrule{1-3}\cmidrule{5-7}\cmidrule{9-11}
$\mathbf{N}\mathfrak{p}$&$\mathfrak{p}$&$a_{\mathfrak{p}}(f)$&&$\mathbf{N}\mathfrak{p}$&$\mathfrak{p}$&$a_{\mathfrak{p}}(f)$&&$\mathbf{N}\mathfrak{p}$&$\mathfrak{p}$&$a_{\mathfrak{p}}(f)$\\\midrule
2&$219w + 2105$&$e - 1$&&3&$4w - 43$&$2e$&&3&$4w - 43$&$-2e + 1$\\
2&$219w - 2324$&$-e$&&3&$4w + 39$&$-2e + 1$&&3&$4w + 39$&$2e$\\
3&$-11066w - 106365$&$-e + 2$&&4&$2$&$3$&&4&$2$&$e - 2$\\
3&$11066w - 117431$&$e + 1$&&5&$-w - 10$&$e - 2$&&5&$-w - 10$&$3$\\
5&$-18w + 191$&$-e$&&5&$w - 11$&$3$&&5&$w - 11$&$e - 2$\\
5&$-18w - 173$&$e - 1$&&7&$54w + 527$&$e - 2$&&7&$54w + 527$&$3$\\
17&$8w + 77$&$4$&&7&$-54w + 581$&$-e + 5$&&7&$-54w + 581$&$e - 2$\\
17&$8w - 85$&$4$&&11&$25w - 269$&$e - 2$&&11&$25w - 269$&$-e + 5$\\
23&$286w - 3035$&$-4e + 3$&&11&$-25w - 244$&$4$&&11&$-25w - 244$&$0$\\
23&$286w + 2749$&$4e - 1$&&17&$-3w + 32$&$0$&&17&$-3w + 32$&$4$\\
41&$-1600w + 16979$&$-e + 5$&&17&$-3w - 29$&$-6e + 3$&&17&$-3w - 29$&$-4e + 5$\\
41&$-1600w - 15379$&$e + 4$&&31&$9w - 97$&$-4e + 5$&&31&$9w - 97$&$-6e + 3$\\\bottomrule
\multicolumn{11}{c}{}\\
\multicolumn{11}{c}{}\\
\toprule
\multicolumn{3}{c}{$D=421,\,\,D'=5$}&\phantom{abc}&\multicolumn{3}{c}{$D=433,\,\,D'=12$}&\phantom{abc}&\multicolumn{3}{c}{$D=461,\,\,D'=29$}\\\cmidrule{1-3}\cmidrule{5-7}\cmidrule{9-11}
$\mathbf{N}\mathfrak{p}$&$\mathfrak{p}$&$a_{\mathfrak{p}}(f)$&&$\mathbf{N}\mathfrak{p}$&$\mathfrak{p}$&$a_{\mathfrak{p}}(f)$&&$\mathbf{N}\mathfrak{p}$&$\mathfrak{p}$&$a_{\mathfrak{p}}(f)$\\\midrule
3&$4w - 43$&$2$&&2&$-w + 11$&$-e$&&4&$2$&$-2$\\
3&$4w + 39$&$2$&&2&$w + 10$&$e$&&5&$-w - 10$&$e$\\
4&$2$&$2e - 1$&&3&$1202w - 13107$&$e - 1$&&5&$-w + 11$&$-e + 1$\\
5&$-w - 10$&$e + 2$&&3&$-1202w - 11905$&$-e - 1$&&9&$3$&$-3$\\
5&$w - 11$&$e + 2$&&11&$-324w - 3209$&$-e - 3$&&17&$w + 11$&$-e + 4$\\
7&$54w + 527$&$-2e + 2$&&11&$-324w + 3533$&$e - 3$&&17&$-w + 12$&$e + 3$\\
7&$-54w + 581$&$-2e + 2$&&13&$94w + 931$&$-3$&&19&$3w - 34$&$-e + 3$\\
11&$25w - 269$&$-4$&&13&$-94w + 1025$&$-3$&&19&$-3w - 31$&$e + 2$\\
11&$-25w - 244$&$-4$&&17&$17152w - 187031$&$-2e - 3$&&23&$-2w + 23$&$-e + 3$\\
17&$-3w + 32$&$-5e+3$&&17&$-17152w - 169879$&$2e - 3$&&23&$-2w - 21$&$e + 2$\\
17&$-3w - 29$&$-5e+3$&&25&$5$&$0$&&41&$w - 13$&$-2e + 2$\\
31&$9w - 97$&$2e+4$&&37&$-12w - 119$&$-3$&&41&$-w - 12$&$2e$\\\bottomrule
\multicolumn{11}{c}{}\\
\multicolumn{11}{c}{}\\
\toprule
\multicolumn{3}{c}{$D=613,\,\,D'=21$}&\phantom{abc}&\multicolumn{3}{c}{$D=677,\,\,D'=13$}&\phantom{abc}&\multicolumn{3}{c}{$D=677,\,\,D'=29$}\\\cmidrule{1-3}\cmidrule{5-7}\cmidrule{9-11}
$\mathbf{N}\mathfrak{p}$&$\mathfrak{p}$&$a_{\mathfrak{p}}(f)$&&$\mathbf{N}\mathfrak{p}$&$\mathfrak{p}$&$a_{\mathfrak{p}}(f)$&&$\mathbf{N}\mathfrak{p}$&$\mathfrak{p}$&$a_{\mathfrak{p}}(f)$\\\midrule
3&$w - 13$&$e$&&4&$2$&$0$&&4&$2$&$-1$\\
3&$-w - 12$&$-e + 1$&&9&$3$&$-4$&&9&$3$&$-3$\\
4&$2$&$0$&&13&$-w + 13$&$-e + 1$&&13&$-w + 13$&$e + 2$\\
7&$8w + 95$&$2$&&13&$-w - 12$&$e$&&13&$-w - 12$&$-e + 3$\\
7&$8w - 103$&$2$&&25&$5$&$-7$&&25&$5$&$-3$\\
17&$33w + 392$&$-e + 5$&&37&$-w - 11$&$-4e - 1$&&37&$-w - 11$&$e - 3$\\
17&$33w - 425$&$e + 4$&&37&$-w + 12$&$4e - 5$&&37&$-w + 12$&$-e - 2$\\
19&$-9w - 107$&$3$&&41&$w - 15$&$-e + 9$&&41&$w - 15$&$3e$\\
19&$9w - 116$&$3$&&41&$w + 14$&$e + 8$&&41&$w + 14$&$-3e + 3$\\
25&$5$&$-6$&&49&$7$&$-3$&&49&$7$&$-10$\\
29&$-w + 14$&$-2e + 7$&&59&$w - 11$&$-2e + 5$&&59&$w - 11$&$-3e + 6$\\
29&$w + 13$&$2e + 5$&&59&$-w - 10$&$2e + 3$&&59&$-w - 10$&$3e + 3$\\\bottomrule
\end{tabular}
\end{center}
\end{table*}

\addtocounter{table}{-1}

\begin{table*}[t!]
\caption{continued}
\begin{center}\scriptsize
\begin{tabular}{@{}crrccrrccrr@{}} 
\toprule
\multicolumn{3}{c}{$D=677,\,\,D'=85$}&\phantom{abc}&\multicolumn{3}{c}{$D=709,\,\,D'=5$}&\phantom{abc}&\multicolumn{3}{c}{$D=797,\,\,D'=8$}\\\cmidrule{1-3}\cmidrule{5-7}\cmidrule{9-11}
$\mathbf{N}\mathfrak{p}$&$\mathfrak{p}$&$a_{\mathfrak{p}}(f)$&&$\mathbf{N}\mathfrak{p}$&$\mathfrak{p}$&$a_{\mathfrak{p}}(f)$&&$\mathbf{N}\mathfrak{p}$&$\mathfrak{p}$&$a_{\mathfrak{p}}(f)$\\\midrule
4&$2$&$-3$&&3&$-59w - 756$&$2e - 1$&&4&$2$&$-3$\\
9&$3$&$-1$&&3&$59w - 815$&$-2e + 1$&&9&$3$&$-3$\\
13&$-w + 13$&$e$&&4&$2$&$0$&&11&$-w + 15$&$3e$\\
13&$-w - 12$&$-e + 1$&&5&$w - 14$&$2e + 1$&&11&$w + 14$&$-3e$\\
25&$5$&$-7$&&5&$-w - 13$&$-2e + 3$&&13&$2w - 29$&$-2e - 1$\\
37&$-w - 11$&$e + 7$&&7&$-16w + 221$&$-2e$&&13&$2w + 27$&$2e - 1$\\
37&$-w + 12$&$-e + 8$&&7&$-16w - 205$&$2e - 2$&&17&$w + 13$&$-2e$\\
41&$w - 15$&$e + 2$&&11&$-547w - 7009$&$-4e + 1$&&17&$w - 14$&$2e$\\
41&$w + 14$&$-e + 3$&&11&$547w - 7556$&$4e - 3$&&25&$5$&$0$\\
49&$7$&$-6$&&19&$6w - 83$&$2e + 4$&&41&$-w - 15$&$2e - 5$\\
59&$w - 11$&$-e - 6$&&19&$6w + 77$&$-2e + 6$&&41&$-w + 16$&$-2e - 5$\\
59&$-w - 10$&$e - 7$&&29&$75w - 1036$&$2e - 3$&&43&$w - 13$&$-e - 4$\\\bottomrule
\multicolumn{11}{c}{}\\
\multicolumn{11}{c}{}\\
\toprule
\multicolumn{3}{c}{$D=797,\,\,D'=29$}&\phantom{abc}&\multicolumn{3}{c}{$D=809,\,\,D'=5$}&\phantom{abc}&\multicolumn{3}{c}{$D=821,\,\,D'=44$}\\\cmidrule{1-3}\cmidrule{5-7}\cmidrule{9-11}
$\mathbf{N}\mathfrak{p}$&$\mathfrak{p}$&$a_{\mathfrak{p}}(f)$&&$\mathbf{N}\mathfrak{p}$&$\mathfrak{p}$&$a_{\mathfrak{p}}(f)$&&$\mathbf{N}\mathfrak{p}$&$\mathfrak{p}$&$a_{\mathfrak{p}}(f)$\\\midrule
4&$2$&$0$&&2&$-219w + 3224$&$-e + 1$&&4&$2$&$-1$\\
9&$3$&$-3$&&2&$-219w - 3005$&$e$&&5&$w - 15$&$0$\\
11&$-w + 15$&$3$&&5&$21796w - 320869$&$e$&&5&$-w - 14$&$0$\\
11&$w + 14$&$3$&&5&$-21796w - 299073$&$-e + 1$&&7&$-6w - 83$&$e - 1$\\
13&$2w - 29$&$e + 3$&&7&$-18w - 247$&$2e - 2$&&7&$6w - 89$&$-e - 1$\\
13&$2w + 27$&$-e + 4$&&7&$18w - 265$&$-2e$&&9&$3$&$-3$\\
17&$w + 13$&$e + 1$&&9&$3$&$-4$&&19&$5w - 74$&$e - 5$\\
17&$w - 14$&$-e + 2$&&13&$-4w - 55$&$-3e + 2$&&19&$5w + 69$&$-e - 5$\\
25&$5$&$-6$&&13&$4w - 59$&$3e - 1$&&23&$-w - 13$&$-e - 3$\\
41&$-w - 15$&$-e + 6$&&19&$140w + 1921$&$e - 5$&&23&$-w + 14$&$e - 3$\\
41&$-w + 16$&$e + 5$&&19&$-140w + 2061$&$-e - 4$&&29&$11w - 163$&$-2e - 3$\\
43&$w - 13$&$2e - 5$&&23&$2926w - 43075$&$-3e + 6$&&29&$-11w - 152$&$2e - 3$\\\bottomrule
\multicolumn{11}{c}{}\\
\multicolumn{11}{c}{}\\
\toprule
\multicolumn{3}{c}{$D=853,\,\,D'=21$}&\phantom{abc}&\multicolumn{3}{c}{$D=929,\,\,D'=13$}&\phantom{abc}&\multicolumn{3}{c}{$D=997,\,\,D'=13$}\\\cmidrule{1-3}\cmidrule{5-7}\cmidrule{9-11}
$\mathbf{N}\mathfrak{p}$&$\mathfrak{p}$&$a_{\mathfrak{p}}(f)$&&$\mathbf{N}\mathfrak{p}$&$\mathfrak{p}$&$a_{\mathfrak{p}}(f)$&&$\mathbf{N}\mathfrak{p}$&$\mathfrak{p}$&$a_{\mathfrak{p}}(f)$\\\midrule
3&$-w + 15$&$-e + 1$&&2&$561w - 8830$&$-e + 1$&&3&$-7w - 107$&$e$\\
3&$-w - 14$&$e$&&2&$561w + 8269$&$e$&&3&$-7w + 114$&$-e + 1$\\
4&$2$&$0$&&5&$-4w - 59$&$-e + 1$&&4&$2$&$0$\\
19&$9w + 127$&$5$&&5&$4w - 63$&$e$&&13&$3w + 46$&$2e + 2$\\
19&$-9w + 136$&$5$&&9&$3$&$3$&&13&$-3w + 49$&$-2e + 4$\\
23&$19w - 287$&$e + 1$&&11&$-8342w + 131301$&$2e - 3$&&19&$4w - 65$&$e + 1$\\
23&$19w + 268$&$-e + 2$&&11&$8342w + 122959$&$-2e - 1$&&19&$4w + 61$&$-e + 2$\\
25&$5$&$3$&&19&$-50w - 737$&$e - 2$&&23&$-w - 16$&$6$\\
31&$w - 14$&$-3$&&19&$50w - 787$&$-e - 1$&&23&$-w + 17$&$6$\\
31&$-w - 13$&$-3$&&23&$-42832w + 674165$&$4e - 4$&&25&$5$&$-4$\\
41&$8w - 121$&$-3e + 3$&&23&$42832w + 631333$&$-4e$&&31&$80w + 1223$&$-2e + 1$\\
41&$8w + 113$&$3e$&&29&$2w + 29$&$-2e + 6$&&31&$80w - 1303$&$2e - 1$\\\bottomrule
\end{tabular}
\end{center}
\end{table*}

\clearpage

\begin{small}
\subsection*{Acknowledgements} 
We thank Fred Diamond and Haluk \c{S}eng\"un for several helpful email
exchanges and discussions, and Noam Elkies, Neil Dummigan, Matthias
Sch\"utt and the anonymous referee for useful comments on earlier
drafts of this paper. We are also thankful to Florian Bouyer and Marco
Streng for many helpful exchanges, and for kindly allowing us to use
their reduction package.  In the early stages of this project, the
first-named author spent some time at the Max-Planck Institute for
Mathematics in Bonn. He would like to express his gratitude for their
hospitality and support.
\end{small}

\end{document}